\documentclass[reqno, 12pt]{amsart}
\usepackage[T1]{fontenc}    
\usepackage[utf8]{inputenc} 
\usepackage{lmodern}       
\usepackage{mathrsfs}
\usepackage{amscd}
\usepackage{amsmath}
\usepackage{latexsym}
\usepackage{amsfonts}
\usepackage{amssymb}
\usepackage{amsthm}
\usepackage{graphicx}
\usepackage{hyperref}
\usepackage{makecell}
\usepackage{array}
\usepackage{booktabs}
\usepackage{multirow}
\usepackage{color,xcolor}
\usepackage{comment}

\newtheorem{theorem}{Theorem}[section]

\newtheorem{corollary}[theorem]{Corollary}
\newtheorem{lemma}[theorem]{Lemma}

\newtheorem{remark}[theorem]{Remark}

\numberwithin{equation}{section}

\title{Uniqueness result for semi-linear wave equations with sources} 
\author[Qiu]{Dong Qiu}
\address{School of Mathematical Science, Zhejiang University}
\email{qiudong@zju.edu.cn}

\author[Xu]{Xiang Xu}
\address{School of Mathematical Sciences, and Center for Interdisciplinary Applied Mathematics, Zhejiang University}
\email{xxu@zju.edu.cn}

\author[Ye]{Yeqiong Ye}
\address{School of Mathematical Science, Zhejiang University}
\email{yeyeqiong@zju.edu.cn}

\author[Zhou]{Ting Zhou}
\address{School of Mathematical Science, Zhejiang University}
\email{ting\_zhou@zju.edu.cn}
\thanks{\\ \emph{Date: }\today\\
\emph{2020 MSN Mathematics Subject Classification.} 35R30, 65N21}

\begin{document}

\begin{abstract}
This paper addresses the inverse problem of simultaneously recovering multiple unknown parameters for semilinear wave equations from boundary measurements. We consider an initial-boundary value problem for a wave equation with a general semilinear term and an internal source. The inverse problem is to determine the nonlinear coefficients (potentials), the source term, and the initial data from the Dirichlet-to-Neumann (DtN) map.

Our approach combines higher-order linearization and the construction of complex geometrical optics (CGO) solutions. 
The main results establish that while unique recovery is not always possible, we can precisely characterize the gauge equivalence classes in the solutions to this inverse problem. For a wave equation with a polynomial nonlinearity of degree 
$n$, we prove that only the highest-order coefficient can be uniquely determined from the DtN map; the lower-order coefficients and the source can only be recovered up to a specific gauge transformation involving a function $\psi$. Furthermore, we provide sufficient conditions under which unique determination of all parameters is guaranteed. We also extend these results to various specific non-polynomial nonlinearities, demonstrating that the nature of the nonlinearity critically influences whether unique recovery or a gauge symmetry is obtained.

\end{abstract}

\maketitle

\section{Introduction}
\subsection{Statement of the problem} 
Let $\Omega \subset {\mathbb{R}^d}$ ($d\geq2$) be a bounded domain with smooth boundary $\partial \Omega$. 
For $T>0$, we denote $Q:=(0,T)\times\Omega$ and $\Sigma:=(0,T)\times\partial\Omega$. 
We study the inverse problem associated with the semilinear wave equation as the following
\begin{equation} \label{eq1.1}
\left\{   
 \begin{aligned}
&\partial_{tt} u(x,t) - \Delta u(x,t) + \sum_{i=1}^{n} q_i(x,t) u(x,t)^i= F(x,t) \ && \text{in } Q,\\
&u(x,t) = f(x,t) \ && \text{on } \Sigma,  \\
&u(x,0) =g(x), \quad \partial_t u(x,0) = h(x) \ && \text{in } \Omega.  
\end{aligned}   
\right.
\end{equation}
where $n\geq2$. 
In Section 2, we first establish the local well-posedness of \eqref{eq1.1} for the initial boundary conditions $(g, h, f)$ in a neighborhood $U_\epsilon$ (defined in Section 2) of $(g_0=u_0|_{t=0}, h_0=\partial_t u_0|_{t=0}, f_0=u_0|_\Sigma)$ where $u_0$ is an existing solution to the equation
\[\partial_{tt} u_0(x,t) - \Delta u_0(x,t) + \sum_{i=1}^{n} q_i(x,t) u_0(x,t)^i= F(x,t) \quad \text{in } Q.\]
Then for $(g,h,f)\in U_\epsilon$, we define the Dirichlet-to-Neumann (DtN) map 
\begin{equation*}   
\Lambda_{\vec q,F,g,h}:=\Lambda_{q_1,q_2,\dots,q_n,F,g,h}:\,f\,\mapsto\,\partial_{\nu}u_f|_{\Sigma},
\end{equation*}
where $\nu$ is the unit outer normal vector on $\partial \Omega$ and $u_f$ is the unique solution to \eqref{eq1.1}.
In this paper, we consider the simultaneous determination of the coefficients $\vec q:=(q_1, \ldots, q_n)$, the source $F$ and the initial data $g$ and $h$ from the DtN map $\Lambda_{\vec q,F,g,h}$. 

Motivated by the challenge of inverting wave propagation data, researchers have extensively studied the corresponding problem's uniqueness and stability.
To solve inverse problems for linear hyperbolic equations, the boundary control method was developed by Belishev in \cite{belishev1987approach}. The boundary control method relies on a sharp unique continuation result for the wave equation that was proved by Tataru \cite{MR1326909}. However, this result is only valid if the coefficients  are time-independent or the lower order terms depend analytically on time \cite{eskin2007inverse}. 

For linear wave equations, numerous studies have addressed this problem using Dirichlet-to-Neumann map observations . In \cite{rakesh_uniqueness_1988}, Rakesh and Symes proved that the DtN map determines uniquely a time-independent potential and  in \cite{isakov_inverse_1991} Isakov considered the determination of a time-independent potential and a damping coefficient. The uniqueness from partial boundary observations has been considered in \cite{eskin_inverse_2007, kian_unique_2017}. 
We remark that for the linear equation $(\partial_{tt}-\Delta)u+Vu=0$, the problem of recovering a time-dependent potential $V$ from the DtN map is still open in general. 
For this problem, Kian and Oksanen  \cite{kian_recovery_2019} proved unique determination of potential given the Cauchy data set or certain subset on the whole boundary $\partial \overline{Q}$ on a compact Riemannian manifold. The unique recovery of a time-dependent magnetic vector-valued potential and an electric scalar-valued potential on a Riemannian manifold has been treated in \cite{feizmohammadi_recovery_2021}. We also mention that the stability issue related to this problem  has been studied in \cite{aicha_stability_2015,bao_sensitivity_2014,bellassoued_stability_2009,bellassoued_stability_2011,kian_stability_2016,stefanov_stability_1998}. \\
A recent observation by Kurylev, Lassas and Uhlmann \cite{kurylev_inverse_2018} is that a nonlinearity in the studied equation can be used as a beneficial tool in the corresponding inverse problem. By developing the higher order linearization method and exploring the propagation and nonlinear interaction of singularities in distorted plane waves, 
 they proved that the source-to-solution map determines the global topology, the differentiable structure and the conformal class of the metric of a globally hyperbolic 3+1-dimensional Lorentzian manifold. Subsequently, it was shown in \cite{ hintz_dirichlet--neumann_2022} 
that using the boundary DtN map one can recover the metric and the nonlinear coefficient.  
For general operators, we refer to \cite{oksanen_inverse_2024} for uniqueness results of determining coefficients in nonlinear real principal type equations. For applications in other physical models, such as Westervelt equations, JMGT equations, nonlinear elastic wave equations, and nonlinear progressive wave equations, have been extensively studied in \cite{acosta_nonlinear_2022,fu_inverse_2024,jiang_inverse_2025,li_inverse_2024,uhlmann_determination_2024,uhlmann_inverse_2023,lai_partial_2024,feizmohammadi_recovery_2022,wang_inverse_2019,nakamura_inverse_2020,hintz_inverse_2022}.  \\
Simultaneous recovery of sources (initial source or equation source) and medium coefficients have been studied extensively due to their importance in applications. For example, motivated by the coupled-physics high resolution imaging method such as  photoacoustic tomography, the inverse problem of simultaneous recovery of the wave speed and the initial conditions for linear wave equations has been studied in \cite{MR3180684,xu2006photoacoustic}. For a semilinear wave equation, the unique determination of the nonlinearity and the initial sources is recently shown in \cite{lin_determining_2024}. For the inhomogeneous semilinear elliptic equation $\Delta u+a(x,u)=F$, the authors of \cite{liimatainen_uniqueness_2024} have shown the uniqueness of the nonlinear coefficients $\partial_u^k a(x,u_0)$ and the source $F$ up to gauge symmetry.  For a semilinear parabolic equation, the authors of \cite{kian_determining_2024} have shown  determination of gauge class of the semilinear term and  gauge invariance break in several situations.
We refer the reader to \cite{sabarreto_recovery_2024,kian_determination_2021,kian_recovery_2023} for recent advances in the uniqueness of inversion for nonlinear hyperbolic and parabolic systems.
The research of stable inversion for nonlinear hyperbolic equations is rather recent. 
The authors of \cite{lassas_uniqueness_2022,lassas_stability_2025} have established the stability of H\"older type to recover nonlinear coefficients from the DtN map.
The authors of \cite{chen_stable_2025} developed a methodology to establish stability estimates for {inversion of coefficients and nonlinearities}.

In this paper, we investigate the simultaneous recovery of unknown potentials, source terms, and initial data. Our approach combines the higher order linearization, approximation of CGO solutions and observability inequalities to establish the main results. Motivated by \cite{liimatainen_uniqueness_2024}, we consider the higher order linearization method in the neighborhood of a nonzero solution $u_0$. This implies that the potential of linearized equation contains the unknown solution, making the simultaneous recovery of multiple potentials and sources more challenging.  \\
\\

It is clear that in general the uniqueness is not guaranteed. When the equation is linear, let $u$ be the solution to 
\begin{equation}\label{eq1.2}
\left\{
\begin{aligned}
&(\partial_{tt} - \Delta) u + q u = F && \text{in } Q,  \\
&u = f && \text{on } \Sigma, \\ 
&u = \partial_t u = 0 && \text{on } \{t=0\} \times \Omega. 
\end{aligned}    
\right.
\end{equation}
and let $\psi(t,x)$ be an arbitrary nonzero $C^2\left(\overline Q\right)$-function satisfying $\psi(0,x)=\partial_t\psi(0,x)=0$ in $\Omega$ and $\psi|_{\Sigma}=\partial_\nu\psi|_{\Sigma}=0$.
Then we have that $\tilde u=u+\psi$ satisfies
\begin{equation*}
(\partial_{tt}-\Delta)\tilde u + q \tilde u 
=\widetilde F:=F+(\partial_{tt}-\Delta) \psi +q \psi,
\end{equation*}
which implies
\[\Lambda_{q,F,0,0}=\Lambda_{q,\widetilde F,0,0}.\]
Our main result is to characterize analogous gauge equivalence classes for semilinear wave equations.

Before we present our main results, let us introduce the following notations and definitions. For nonnegative integers $m,k$, define the following function spaces: 
\begin{align*}
&H^k(\partial\Omega) =  \left\{ v \in L^2(\partial\Omega)\ \middle|  \ D^\beta v \in L^2(\partial\Omega), \forall\ \beta\in\mathbb{N}^d, |\beta| \leq k \right\}, \\
& H^m_0(\Omega) = \left\{  u \in H^m(\Omega)\ \middle|  \ D^{\alpha} u = 0 \text{ on } \partial \Omega \text{ in the trace sense} , \forall\ \alpha\in\mathbb{N}^d, |\alpha| \leq m-1 \right\}, \\
   &H^m\left((0,T) \times \partial\Omega\right) = \left\{ u \in L^2\bigl((0,T) \times \partial\Omega\bigr) \ \middle| \ \partial_t^j \partial_x^\beta u \in L^2\bigl((0,T) \times \partial\Omega\bigr),\ \forall\, j + |\beta| \leq m \right\},   \\
   & H_0^m\left((0,T); H^k(\partial\Omega)\right) = \left\{ u \in L^2\left((0,T); H^k(\partial\Omega)\right) \ \middle| \ 
        \begin{aligned}
            \ \partial_t^j u \in L^2\left((0,T); H^k(\partial\Omega)\right), \forall j \leq m  \\
             \ \partial_t^j u(0,\cdot) = 0 \  \text{in the trace sense} ,\forall j < m
        \end{aligned}
    \right\}, \\
&O_m:=  \left \{f\in H^m(\Sigma)\ \middle|\  f \in H_0^{m-k}\left((0,T);H^k(\partial \Omega)\right), \  \text{for} \  k=0,1,\dots,m-1  \right\},   
\end{align*}

and
\begin{align}\label{def_G}
G_{m+1}:=H_0^{m+1}(\Omega) \times H_0^{m}(\Omega) \times O_{m+1}
\end{align}
with the norm
\begin{align*}
||(g,h,f)||_{G_{m+1}}=||g||_{H^{m+1}(\Omega)}+||h||_{H^{m}(\Omega)}+||f||_{H^{m+1}(\Sigma)}.
\end{align*}
The convenient spaces for solutions of the wave equation are called \textit{energy spaces} $E^m$, defined as  
\begin{equation*}
E^m=\underset{0\leq k \leq m}{\bigcap} C^k\left([0,T];H^{m-k}(\Omega)\right),
\end{equation*}
equipped with the norm 
\begin{equation*}
||u||_{E^m}=\underset{0<t<T}{\sup}\sum\limits_{k=0}^{m}||\partial_{t}^{k}u(\cdot,t))||_{H^{m-k}(\Omega)}.
\end{equation*}
The space $E^m$ is an algebra if $m >d+1$ (see \cite{choquet-bruhat_general_2009}). Moreover, it satisfies the norm estimate
\begin{equation*}
||uv||_{E^m}\leq C_m ||u||_{E^m}||v||_{E^m},  \quad \text{for all } u,v\in E^m,
\end{equation*}
where $C_m$ is a constant depending on $m$. \\

\subsection{Main Results}
In the following, we assume that $T^*$ is a constant, $T^*>2 \operatorname{diam}(\Omega)$ and $t_1$, $t_2$ are two constants such that $T^* < t_1<t_2<T-T^*$. We also assume that $n\geq2$ and  $m>d+1$.
\begin{theorem}\label{lem_1.3}
 Assume that $a_j(t,x,z): \mathbb{R} \times \overline{\Omega} \times \mathbb{R} \to \mathbb{R}$ satisfies $a_j(t,x,0)=0$ for all $x \in \Omega$ and the map $z \mapsto a_j(\cdot,\cdot,z)$ is analytic with values in $E^m$ for $j=1,2$.  Assume that    $supp\left(a_j(t,x,z)\right)\subseteq [t_1,t_2]\times \overline{\Omega} \times \mathbb{R}$ for $j=1,2$, and $F_j \in E^{m}$ with {$\partial_t^k F_j(0,\cdot)\in H_0^{m-k}(\Omega)$} and $supp(F_j) \subset [t_1,T] \times \Omega$, $g \in H_0^{m+1}(\Omega)$, $h \in H_0^m(\Omega)$ for $j=1,2$ and $k=0,1,\cdots,m-2$. Let $\Lambda_{ a_j,F_j}$ be the DtN map of the equation 
\begin{equation} \label{eqa}
\left\{
\begin{aligned}
&\partial_{tt} u_j - \Delta u_j + a_j(t,x,u_j) = F_j
&& \text{in } Q, \\
&u_j = f
&& \text{on } \Sigma, \\
&u_j =g, \quad \partial_t u_j = h
&& \text{on }\{t=0\} \times \Omega.
\end{aligned}
\right.
\end{equation}
Suppose that there is an open set $\mathcal N \subset O_{m+1}$ such that 
\begin{align}
\Lambda_{ a_1,F_1}(f)=\Lambda_{ a_2,F_2}(f) \quad \text{for any} \ f \in \mathcal N. 
\end{align}
Then for any $f_0 \in \mathcal N$, we have
\begin{align*}
\partial_z^ka_1(t,x,u_{0,1})=\partial_z^ka_2(t,x,u_{0,2}) \quad \text{in} \ [t_1, t_2] \times  \Omega, \text{ for any}\ k \in  \mathbb{N}.
\end{align*}
Here $u_{0,1}$ and $u_{0,2}$ are the solutions to \eqref{eqa} with boundary condition $u_{0,j}|_{\Sigma}=f_{0}$. 
\end{theorem}
\begin{theorem} \label{thm1.1}
Assume that  $\vec q_j(t,x):=\left(q_{1,j}(t,x),\ldots,q_{n,j}(t,x)\right)\in (E^{m})^n$ with $supp(q_{i,j})\subseteq [t_1,t_2]\times \overline{\Omega}$ for $i=1,2,... ,n$ and $j=1,2$, and $F_j \in E^{m}$ with {$\partial_t^k F_j(0,\cdot)\in H_0^{m-k}(\Omega)$} and $supp(F_j) \subset [t_1,T] \times \Omega$, $g_j \in H_0^{m+1}(\Omega)$, $h_j \in H_0^m(\Omega)$ for $j=1,2$ and $k=0,1,\cdots,m-2$. Let $\Lambda_{\vec q_j,F_j,g_j,h_j}$ be the DtN map of the equation 
\begin{equation}\label{eq_nonlinear}
\left\{
\begin{aligned}
&\partial_{tt} u_j - \Delta u_j + \sum_{i=1}^{n} q_{i,j} u_j^i = F_j
&& \text{in } Q, \\
&u_j = f
&& \text{on } \Sigma, \\
&u_j =g_j, \quad \partial_t u_j = h_j
&& \text{on }\{t=0\} \times \Omega.
\end{aligned}
\right.
\end{equation}
Suppose that there is an open set $\mathcal N \subset O_{m+1}$ such that 
\begin{align}\label{eq:thm1_DN}
\Lambda_{\vec q_1,F_1,g_1,h_1}(f)=\Lambda_{\vec q_{2},F_2,g_2,h_2}(f) \quad \text{for any} \ f \in \mathcal N. 
\end{align}
Then there exists $\psi \in E^{m+1}$ with $\psi|_\Sigma=\partial_{\nu} \psi|_ \Sigma=0$ and $\psi=0$ in $[0,t_1] \times \Omega$ such that 
\begin{equation} \label{eq:thm1_result}
\left\{
\begin{aligned}
    q_{n,1} &= q_{n,2}  && \ \text{in } \ [t_1,t_2]\times \Omega  , \\
    q_{n-1,1} &= q_{n-1,2}+\binom{n}{1}q_{n,2} \psi && \ \text{in } \ [t_1,t_2]\times \Omega , \\
    &\vdots \\
    q_{n-k,1} &= \sum_{l=0}^{k}\binom{n-k+l}{l} q_{n-k+l,2} \psi^l && \ \text{in } \ [t_1,t_2]\times \Omega , \\
    &\vdots \\
    F_1 &= F_2-(\partial_{tt}-\Delta) \psi -\sum_{i=1}^{n} q_{i,2} \psi^i  && \ \text{in } \ [t_1,t_2]\times \Omega, \\
    g_1&=g_2 \  \text{and } h_1=h_2  && \ \text{in }  \Omega.
\end{aligned}
\right.
\end{equation}
Conversely, if the condition \eqref{eq:thm1_result} holds for some $\psi \in E^{m+1}$ with $\psi|_\Sigma=\partial_{\nu} \psi|_\Sigma=0$ and  $\psi=0$ in $[0,t_1] \times \Omega$, then
\begin{equation*}
\Lambda_{\vec q_1,F_1,g,h}(f)=\Lambda_{\vec q_2,F_2,g,h}(f) \ \  \text{for any} \ f \in \mathcal N.
\end{equation*}
\end{theorem}

\begin{remark}
For $n=2$, the condition \eqref{eq:thm1_result} implies that 
\[
\left\{
\begin{array}{ll}
    q_{2,1}=q_{2,2}:=q_2 &\text{in} \ [t_1,t_2]\times \Omega, \\ 
    q_{1,1}=q_{1,2}+2q_{2} \psi  &\text{in}  \ [t_1,t_2]\times \Omega, \\
    F_1=F_2-(\partial_{tt}-\Delta) \psi - q_{1,2} \psi-q_2 \psi^2   &\text{in} \ [t_1,t_2]\times \Omega.
\end{array}
\right.
\]
For $n=3$, the condition \eqref{eq:thm1_result} implies that 
\[
\left\{
\begin{array}{ll}
    q_{3,1}=q_{3,2}:=q_3   &\text{in} \ [t_1,t_2]\times \Omega, \\ 
    q_{2,1}=q_{2,2}+3q_{3} \psi   &\text{in} \ [t_1,t_2]\times \Omega, \\ 
    q_{1,1}=q_{1,2}+2q_{2,2}\psi+3q_3 \psi^2   &\text{in} \ [t_1,t_2]\times \Omega,\\
    F_1=F_2-(\partial_{tt}-\Delta) \psi - q_{1,2} \psi -q_{2,2}\psi^2-q_3 \psi^3   &\text{in} \ [t_1,t_2]\times \Omega.
\end{array}
\right.
\]
We note that this result shows that only the highest order coefficient can be uniquely determined.
\end{remark}

\begin{corollary} \label{corollary1}
Assume that $q_j \in E^{m+1}$ with $supp(q_{j})\subseteq [t_1,t_2]\times \overline{\Omega}$, $F_j \in E^{m}$ with {$\partial_t^k F_j(0,\cdot)\in H_0^{m-k}(\Omega)$} and $supp(F_j) \subset [t_1,T] \times \overline{\Omega}$, $g_j \in H_0^{m+1}(\Omega)$, $h_j \in H_0^m(\Omega)$ for $j=1,2$ and $k=0,1,\cdots,m-2$, $u_j$ is the solution to 
\begin{equation}
\left\{
\begin{aligned}
&\partial_{tt} u_j - \Delta u_j +  q_{j} u_j = F_j
&& \text{in } Q, \\
&u_j = f
&& \text{on } \Sigma, \\
&u_j =g_j, \quad \partial_t u_j = h_j
&& \text{on }\{t=0\} \times \Omega.
\end{aligned}
\right.
\end{equation}
Suppose that there is an open set $\mathcal N \subset O^{m+1}$ such that 
\begin{equation*}
\Lambda_{q_1,F_1,g_1,h_1}(f)=\Lambda_{q_2,F_2,g_2,h_2}(f)  \quad \text{for }  f\in \mathcal N.
\end{equation*}
Then the potential and source can be determined 
\begin{equation*}
q_1=q_2:=q  \quad \text{in} \ [t_1,t_2] \times \Omega,  \quad g_1=g_2, \ h_1=h_2 \ \text{in }   \Omega,
\end{equation*}
and there exists $\psi \in E^{m+1}$, $\psi|_{\Sigma}=\partial_{\nu}  \psi|_{\Sigma}=0 $ and $\psi=0$ in $[0,t_1] \times \Omega$ such that 
\begin{equation*}
F_1=F_2+\partial_{tt} \psi-\Delta \psi+q \psi.
\end{equation*}
\end{corollary}

\begin{theorem}\label{thm_2}
Assume that  $\vec q_j(t,x):=\left(q_{1,j}(t,x),\ldots,q_{n,j}(t,x)\right)\in (E^{m})^n$ with $supp(q_{i,j})\subseteq [t_1,t_2]\times \overline{\Omega}$ for $i=1,2,...,n$ and $j=1,2$, and $F_j \in E^{m}$ with {$\partial_t^k F_j(0,\cdot)\in H_0^{m-k}(\Omega)$} , $g_j \in H_0^{m+1}(\Omega)$, $h_j \in H_0^m(\Omega)$ for $j=1,2$ and $k=0,1,2,\cdots,m-2$. Let $\Lambda_{\vec q_j,F_j,g_j,h_j}$ be the DtN map of the equation 
\begin{equation}\label{eq_nonlinear 2}
\left\{
\begin{aligned}
&\partial_{tt} u_j - \Delta u_j + \sum_{i=1}^{n} q_{i,j} u_j^i = F_j
&& \text{in } Q, \\
&u_j = f
&& \text{on } \Sigma, \\
&u_j =g_j, \quad \partial_t u_j = h_j
&& \text{on }\{t=0\} \times \Omega.
\end{aligned}
\right.
\end{equation}
Suppose that there is an open set $\mathcal N \subset O_{m+1}$ such that 
\begin{align}\label{eq:thm1_DN 2}
\Lambda_{\vec q_1,F_1,g_1,h_1}(f)=\Lambda_{\vec q_{2},F_2,g_2,h_2}(f) \quad \text{for any} \ f \in \mathcal N. 
\end{align}
Assumption(1): $ supp (F_j) \subset [t_1,{T}] \times \Omega$, $q_{n-1,1}=q_{n-1,2}$ in $Q$, either $q_{n,1}\neq0$ or $q_{n,2} \neq 0$ in $Q$. Then there exists $\psi \in E^{m+1}$ with $\psi|_\Sigma=\partial_{\nu} \psi|_\Sigma=0$ and  $\psi=0$ in $[0,t_2] \times \Omega$ such that 
\begin{align*}
&F_1=F_2 , \quad  q_{i,1}=q_{i,2}\quad &&\text{in} \ [t_1,t_2]\times \Omega \ \  \text{for} \ i=1,2, \dots,n, \\
&F_1=F_2+\partial_{tt} \psi-\Delta \psi &&\text{in} \ (t_2,T)\times \Omega,  \\
&g_1=g_2 \quad \text{and} \quad h_1=h_2  \quad \ \ &&\text{in}  \  \Omega.
\end{align*}
Assumption(2):  $F_1=F_2$ in $Q$. Then all the coefficients and initial data are determined:
\begin{align*}
&g_1=g_2 \quad \text{and} \quad h_1=h_2  \quad \ \ &&\text{in}  \  \Omega, \\
& q_{i,1}=q_{i,2} && \text{in} \ [t_1,t_2]\times \Omega \ \  \quad \text{for} \ i=1,2, \dots,n.
\end{align*}

\end{theorem}

\begin{corollary}[Special cases of nonlinearities] \label{corollary apply}
 Assume that $p_j(t,x),q_j(t,x) \in E^m$, $supp(p_j)$, $supp(q_j) \subseteq [t_1,t_2]\times \overline{\Omega} $ for $j=1,2$, and $F_j \in E^{m}$ with {$\partial_t^k F_j(0,\cdot)\in H_0^{m-k}(\Omega)$} and $supp(F_j) \subset [t_1,t_2] \times \Omega$, $g \in H_0^{m+1}(\Omega)$, $h \in H_0^m(\Omega)$ for $j=1,2$ and $k=0,1,\cdots,m-2$. Additionally suppose that
 \begin{align*}
 \text{Case} \ 1 : & a_j(t,x,z)=q_j \mathrm{e}^{z},
 \ \ \ \ \ \ \text{Case} \ 2 :  a_j(t,x,z)=q_j \sin (z), \\
 \text{Case}  \ 3: & a_j(t,x,z)=q_j z \sin z, 
 \ \text{Case} \  4:  a_j(t,x,z)=p_j \sin z+q_j\mathrm{e} ^{z}, \\
 \text{Case} \ 5 : & a_j(t,x,z)=p_j \sin z+q_j \cos z, \\
 \text{Case} \ 6 : & a_j(t,x,z)=p_j z \mathrm{e}^{z} + \sum_{k=1}^n q_{k,j}z^k,  \ \text{with} \   p_j \neq 0  \ \text {in} \ [t_1,t_2] \times \Omega \ \text{for}  \ j=1,2.
 \end{align*}
 Let $\Lambda_{ a_j,F_j}$ be the DtN map of the equation 
\begin{equation} 
\left\{
\begin{aligned}
&\partial_{tt} u_j - \Delta u_j + a_j(t,x,u_j) = F_j
&& \text{in } Q, \\
&u_j = f
&& \text{on } \Sigma, \\
&u_j =g, \quad \partial_t u_j = h
&& \text{on }\{t=0\} \times \Omega.
\end{aligned}
\right.
\end{equation}
Suppose that there is an open set $\mathcal N \subset O_{m+1}$ such that 
\begin{align}
\Lambda_{a_1,F_1}(f)=\Lambda_{ a_2,F_2}(f) \quad \text{for any} \ f \in \mathcal N. 
\end{align}
Then we have:  \\
Case 1 Gauge symmetry:
\begin{equation*}
q_1=q_2 e^{\psi}  \quad \text{and} \quad F_2-F_1=\partial_{tt} \psi -\Delta \psi \ \text{in} \ [t_1,t_2] \times \Omega.
\end{equation*}
Case 2 Unique determination:
\begin{equation*}
q_1=q_2   \quad \text{and} \quad F_1=F_2 \ \text{in} \ [t_1,t_2] \times \Omega.
\end{equation*}
Case 3  Unique determination:
\begin{equation*}
q_1=q_2   \quad \text{and} \quad F_1=F_2 \ \text{in} \ [t_1,t_2] \times \Omega.
\end{equation*}
Case 4  Unique determination:
\begin{equation*}
p_1=p_2, \ q_1=q_2   \quad \text{and} \quad F_1=F_2 \ \text{in} \ [t_1,t_2] \times \Omega.
\end{equation*}
Case 5 Gauge symmetry:
\begin{equation*}
p_2=p_1 \cos \psi+q_1 \sin \psi, \ q_2=-p_1 \sin \psi + q_1 \cos \psi ,  \quad \text{and} \quad F_2-F_1=\partial_{tt} \psi -\Delta \psi \ \text{in} \ [t_1,t_2] \times \Omega.
\end{equation*}
Case 6  Unique determination:
\begin{equation*}
p_1=p_2, \ q_{k,1}=q_{k,2} \ \text{for} \ k=1,2,\cdots,n   \quad \text{and} \quad F_1=F_2 \ \text{in} \ [t_1,t_2] \times \Omega.
\end{equation*}
\end{corollary}


This paper is organized as follows. In Section \ref{section2}, we establish the local well-posedness for the forward problem. In Section \ref{section3}, we provide the proofs of the main theorems.

\section*{Acknowledgements}
 X Xu is partly supported by National Key Research and Development Program of China (No. 2024YFA1012300), National Natural Science Foundation of China (No. 12525112), and the Open Research Project of Innovation Center of Yangtze River Delta, Zhejiang University. TZ was partially supported by the National Key Research and Development Program of China (No. 2024YFA1012301), the Zhejiang Provincial Basic Public Welfare Research Program [
Grant Number LDQ24A010001], and NSFC Grant 12371426.

\section{Well-posedness of the forward problem}  \label{section2}
In this section, we establish the well-posedness of the initial boundary value problem \eqref{eq1.1} for the semilinear wave equation. 
\begin{lemma}[Well-posedness for linear equations (see \cite{lin_determining_2024} Lemma 3.2)] \label{lem:lin_wp}
Let $m$ be a positive integer and $m > d+1$. Assume that  $g \in H_0^{m+1}(\Omega)$, $h \in H_0^m(\Omega)$, $f \in O_{m+1}$, $q \in E^m$ and $F \in E^m$ with $\partial_t^kF(0,\cdot) \in H_0^{m-k}(\Omega)$ for $k=0,1,\cdots,m-2$. 
Then the equation
\[\left\{
\begin{aligned}
&\partial_{tt} u - \Delta u + qu = F
&& \text{in } Q, \\
&u = f
&& \text{on } \Sigma, \\
&u = g, \ \partial_t u = h
&& \text{on } \{ t=0 \} \times \Omega.
\end{aligned}
\right.\]
has a unique solution $u\in E^{m+1}$ with $\partial_{\nu} u \in H^{m}(\Sigma)$ satisfying 
\begin{align}\label{eq_linear_wp}
&||u||_{E^{m+1}}+ ||\partial_{\nu}u||_{H^{m}(\Sigma)}\leq \notag \\
&C_{m,T}\left(\sum_{k=0}^m||\partial_t^k F||_{L^1((0,T);H^{m-k}(\Omega))}+||g||_{H^{m+1}(\Omega)}+||h||_{H^m(\Omega)}+||f||_{H^{m+1}(\Sigma)}\right), 
\end{align}
where $C_{m,T}$ is a constant depending on $m$ and $T$. 
\end{lemma}

\begin{theorem}
Assume that $m$ is an integer and $m > d+1$. Assume that $a(t,x,z): \mathbb{R} \times \overline{\Omega} \times \mathbb{R} \to \mathbb{R}$ satisfies $a(t,x,0)=0$ for all $x \in \Omega$ and the map $z \mapsto a(\cdot,\cdot,z)$ is analytic with values in $E^m$.
Suppose that for given $(g_0,h_0,f_0)\in G_{m+1}$ and $F\in E^{m}$ satisfies $\partial_t^kF(0,\cdot)\in H_0^{m-k}(\Omega)$ for $k=0,1,\cdots,m-2$ 
there exists a unique solution $u_0{\in E^{m+1}}$ for the equation 
\begin{equation}
\left\{
\begin{aligned}
&\partial_{tt}  u_0 - \Delta u_0 + a(t,x,u_0) = F 
&& \text{in } Q,\\
&u_0 = f_0
&& \text{on } \Sigma, \\
&u_0 =g_0, \ \partial_t u_0 = h_0  
&& \text{on } \{t=0\} \times \Omega.
\end{aligned}
\right.
\end{equation}
Then there exists a sufficiently small $\epsilon >0$ and a constant $C>0$ such that for any $(g,h,f)$ in the set
\begin{align*}
U_{\epsilon} = \left\{ (g,h,f) \in G_{m+1} \,\middle|\, 
\|(g-g_0, h-h_0, f-f_0)\|_{G_{m+1}} < \epsilon
\right\},
\end{align*}
there exists a unique solution $u \in E^{m+1}$ with $\partial_{\nu}u  \in H^{m}(\Sigma)$ satisfying the equation
\[\left\{
\begin{aligned}
& \partial_{tt} u - \Delta u + a(t,x,u) = F 
&& \text{in } Q, \\
& u = f 
&& \text{on } \Sigma, \\
& u =g ,\ \partial_t u = h 
&& \text{on } \{t=0\} \times \Omega.
\end{aligned}
\right.
\]
Moreover, we have
\begin{align}\label{eq:tilde_u_est}
\|u-u_0\|_{E^{m+1}}+\|\partial_{\nu} (u-u_0)\|_{H^{m}(\Sigma)} \leq  C \|(g-g_0, h-h_0, f-f_0)\|_{G_{m+1}}.
\end{align}
\end{theorem}

\begin{proof}
Let $\tilde u=u-u_0$ 
and denote
\begin{equation*}
G(w,u_0):=-\left( a(t,x,u_0+w)-a(t,x,u_0)-\partial_z a(t,x,u_0) w \right).
\end{equation*}
Then $\tilde u$ should solve
\begin{equation}\label{eq:tilde_u}
\left\{
\begin{aligned}
& \partial_{tt} \tilde{u} - \Delta \tilde{u} +\partial_za(t,x,u_0) \tilde u= G(\tilde u, u_0) 
&& \text{in } Q,\\
& \tilde{u} = f-f_0
&& \text{on } \Sigma, \\
& \tilde{u} =g-g_0, \ \partial_t \tilde{u} = h-h_0
&& \text{on } \{t=0\}\times\Omega.
\end{aligned}
\right.
\end{equation}

Given $w\in B_\rho:=\left\{ u\in E^{m+1}| \  \ ||u||_{E^{m+1} }< \rho \right \}$, when  $u_0\in E^{m+1}$, since $E^m$ is an algebra and $a(t,x,z)$ is analytic with $z$ , one can show that for $\rho$ small enough,  by Taylor's theorem, there exist $\theta \in (0,1)$ such that
\begin{align}\label{eq：G-norm}
    \|G(w,u_0)\|_{E^m}&= \|a(t,x,u_0+w)-a(t,x,u_0)-\partial_z a(t,x,u_0) w  \|_{E^m} \notag\\
    &= \left \|\frac{\partial_{z}^2 a(t,x,u_0+\theta w )}{2} w^2\right \|_{E^m}  \notag\\
    &\leq C  \|w^2\|_{E^m} \leq C \|w\|_{E^{m+1}}^2,
\end{align}
and 
\begin{align} \label{partialw}
\| \partial_wG(w,u_0) \|_{E^m}=\left\| \sum_{i=2}^{\infty}\sum_{l=2}^{i} \frac{\partial_z^i a(t,x,u_0)}{i!} \binom{i}{l} l  w^{l-1} u_0^{i-l} \right\|_{E^m} \leq C \|w \|_{E^m}.
\end{align}
Therefore, $G(w,u_0)\in E^m$. Also, 
since $f_0\in O_{m+1}$,
\[
\partial_t^k G(w,u_0)|_{t=0} =0 \quad \text{on } \partial \Omega,
\]
then $\partial_t^k G(w,u_0) \in H_0^{m-k}(\Omega)$ for $k=0,1,\dots,m-2$. Moreover, from \ref{eq：G-norm} we have
\begin{align}\label{eq:G_est}
    \sum_{k=0}^m\|\partial_t^k G(w, u_0)\|_{L^1((0,T);H^{m-k}(\Omega))} &\leq C_{m,T}T \sum_{k=0}^m\|\partial_t^k G(w, u_0)\|_{C((0,T);H^{m-k}(\Omega))} \notag\\
    &= C_{m,T}T\|G(w,u_0)\|_{E^m} \leq C \|w\|_{E^{m+1}}^2 \leq C \rho^2.
\end{align}
Then by Lemma \ref{lem:lin_wp}, the equation
\begin{equation}\label{eq:v_w}
\left\{
\begin{aligned}
& \partial_{tt} v - \Delta v +\partial_za(t,x,u_0) v= G(w, u_0) 
&& \text{in } Q, \\
& v = f-f_0
&& \text{on } \Sigma, \\
& v =g-g_0, \ \partial_t v = h-h_0
&& \text{on } \{t=0\}\times\Omega,
\end{aligned}
\right.
\end{equation}
admits a unique solution $v\in E^{m+1}$ satisfying
\begin{equation}\label{eq:v_norm}
\begin{aligned}
    \|v\|_{E^{m+1}} &+\|\partial_{\nu} v\|_{H^{m}(\Sigma)} \\ 
    &\leq C\left(\epsilon + \sum_{k=0}^m\|\partial_t^k G(w, u_0)\|_{L^1((0,T);H^{m-k}(\Omega))}\right) \leq C(\epsilon+\rho^2).
\end{aligned}
\end{equation}
For $\epsilon=\frac{1}{2C}\rho$ and $\rho$ small enough, one has $v\in B_\rho$.

Next let $w_j\in B_\rho$ for $j=1,2$ and $v_j$ be the solution to \eqref{eq:v_w} with $w$ replaced by $w_j$. Then we have their difference $\tilde v:=v_1-v_2$ satisfies
\begin{equation*}
\left\{
\begin{aligned}
& \partial_{tt} \tilde v - \Delta \tilde v + \partial_za(t,x,u_0) \tilde v= G(w_1, u_0)-G(w_2, u_0)
&& \text{in } Q, \\
& \tilde v = 0
&& \text{on } \Sigma, \\
& \tilde v =\partial_t \tilde v = 0
&& \text{on } \{t=0\}\times\Omega.
\end{aligned}
\right.
\end{equation*}
Therefore by using \eqref{partialw} we obtain
\begin{align*}
    \|\tilde v\|_{E^{m+1}} &\leq C_{m,T}\sum_{k=0}^m\|\partial_t^kG(w_1,u_0)-\partial_t^kG(w_2, u_0)\|_{L^1((0,T);H^{m-k}(\Omega))}\\
    &\leq C_{m,T} T \|G(w_1,u_0)-G(w_2,u_0)\|_{E^m}\\
    & \leq C \| \partial_wG(sw_1+(1-s)w_2,u_0) \|_{E^m}  \| w_1 - w_2 \|_{E^m}\\
    &\leq C \| s w_1+(1-s)w_2 \|_{E^m}  \| w_1 - w_2 \|_{E^m}\leq C\rho \|w_1-w_2\|_{E^{m+1}}.
\end{align*}
For $\rho$ small enough, this implies the map from $w$ to $v$ is a contraction on $B_\rho$. By the Banach fixed-point theorem, there exists a unique fixed point $\tilde u$ which is the solution to \eqref{eq:tilde_u}. Moreover, by \eqref{eq:G_est}, for $\rho$ small enough, one obtains \eqref{eq:tilde_u_est}.
\end{proof}

\begin{remark}
It is not hard to see that the solution map $(g,h,f)\,\mapsto\,u$ is $C^\infty$ Fr\'echet differentiable.
\end{remark}

\section {Uniqueness result} \label{section3}
In this section, we study the inverse problem on determining nonlinearities, source and initial data of semilinear  wave equation by DtN map. Driven by the need to determine initial data, we recall an observability inequality for the following wave equation: 
\begin{equation}\label{eq3.1}
\left\{
\begin{aligned}
&(\partial_{tt}-\Delta)u +a(t,x) u  = 0 && \text{in } Q, \\
&u = 0 && \text{on } \Sigma, \\
&u =g,\quad \partial_t u = h && \text{on } \left\{t=0\right\}\times\Omega. 
\end{aligned}
\right.
\end{equation}
where $a \in L^{\infty}\left((0,T);L^p(\Omega)\right)$ with $p \geq d$ and $(g,h)\in H^1_0(\Omega)\times L^2(\Omega)$. \\
Similar to  \cite{duyckaerts_optimality_2008}, \cite{fu_sharp_2023} and \cite{lu_observability_2013}, one has the following result:
\begin{lemma} \label{Lemma3.1}
For any $T> T^*$, any solution $u \in C\left([0,T];H^1_0(\Omega)\right) \bigcap C^1\left([0,T]; L^2(\Omega)\right)$ to \eqref{eq3.1} satisfies
\begin{equation}
||g||_{H_0^1(\Omega)}+||h||_{L^2(\Omega)} \leq \mathrm{e}^{C\left(1+||a||_{L^{\infty}\left((0,T);L^p(\Omega)\right)}^{\frac{1}{\frac{3}{2}-\frac{d}{p}}}\right)} ||\partial_{\nu}u||_{L^2(\Sigma)}.
\end{equation}
\end{lemma}
Then we recall an approximation property:
\begin{lemma}[Similar to Theorem 5.1 in \cite{lin_determining_2024}]\label{lem:approx}
Assume  that $T>2T^*$, $t_1$ and $t_2$ are two constants satisfying $T^*<t_1<t_2<T-T^*$  and $q\in E^{m+1}$ with $\textrm{supp}\  q\subseteq [t_1, t_2]\times\overline\Omega$. Then for any solution $v\in C([t_1,t_2];L^2(\Omega))\cap C^1([t_1,t_2];H^{-1}(\Omega))$ to 
\[\partial_{tt}v-\Delta v+  q v=0\qquad \textrm{ in }Q,\]
and any $\varepsilon>0$, there exists a solution $V\in C^2(\overline Q)$ to 
\begin{equation}\label{eq:appV}
\left\{
\begin{aligned}
&\partial_{tt}V-\Delta V+ q V=0\qquad&&\textrm{in }Q,\\
&V(x,0)=\partial_tV(x,0)=0  \quad \text{or } \ V(x,T)=\partial_tV(x,T)=0&&\textrm{in }\Omega,
\end{aligned}
\right.
\end{equation}
such that 
\[\|V-v\|_{L^2((t_1,t_2)\times\Omega)}<\varepsilon.\]
\end{lemma}
\begin{proof}

We consider the case where $V(x,T)=\partial_tV(x,T)=0$ in $\Omega$. Similar to the proof of Theorem 5.1 in \cite{lin_determining_2024}, we aim to show
\begin{align*}
    X=\left\{w =V|_{(t_1,t_2)\times \Omega}\ \middle| \ V \in C^2\left(\overline{Q}\right) \textrm{ is a solution to the second case of }\eqref{eq:appV} \right\}
\end{align*}
is dense in
\begin{align*}
    Y=\left\{ v \in C\left([t_1,t_2];L^2(\Omega)\right) \cap C^1\left([t_1,t_2];H^{-1}(\Omega)\right)\ \middle| \ \partial_{tt}v-\Delta v+  q v=0 \textrm{ in } (t_1,t_2)\times \Omega\right\}
\end{align*}
in terms of $L^2\big(\Omega \times (t_1,t_2)\big)$. By the Hahn-Banach theorem, it suffices to verify: If $f \in L^2\left((t_1,t_2)\times \Omega \right)$ satisfies
\begin{equation}\label{eq:w}
    \int_{t_1}^{t_2} \int_\Omega fw dxdt =0\quad \forall w\in X,
\end{equation}
    then
\begin{equation}\label{eq:v}
    \int_{t_1}^{t_2} \int_\Omega fv dxdt =0\quad \forall v\in Y.
\end{equation}
To this aim, let $f$ satisfy \eqref{eq:w} and set
\begin{equation*}
\tilde{f}(x,t)=\left\{
\begin{aligned}
&f(x,t) &&\text{in } (t_1,t_2)\times \Omega, \\
& 0 && \text{in } \left((0,t_1] \cup [t_2,T)\right) \times \Omega.
\end{aligned}
\right.
\end{equation*}
Assume $\tilde{v} \in H_0$ solve the wave equation:
\begin{equation}\label{eq:tilde{v}}
\left\{
    \begin{aligned}
    &(\partial_{tt} - \Delta) \tilde{v} + q \tilde{v} = \tilde{f} &&\text{in } Q, \\
&\tilde{v} = 0 && \text{on } \Sigma, \\
&\tilde{v} =\partial_t \tilde{v} = 0 && \text{on } \{t=0\}\times\Omega.
    \end{aligned}
\right.    
\end{equation}
Then for any solution $V \in C^2\left(\overline{Q}\right)$ to \eqref{eq:appV} and $w=V|_{(t_1,t_2)\times \Omega}$, 
\begin{equation*}
\begin{aligned}
     0 &= \int_{t_1}^{t_2}\int_{\Omega}fw dxdt =\int_Q \tilde{f}V dxdt \\
    &=\int_Q (\partial_{tt} \tilde{v} - \Delta \tilde{v} + q \tilde{v}) V = \int_{\Sigma} \partial_{\nu} \tilde{v} \cdot V dSdt.
\end{aligned}
\end{equation*}
Since $V|_{\Sigma}$ is arbitrary in $C_0^{\infty}\left(0,T;C^{\infty}(\partial \Omega)\right)$ and the solution $V$ to \eqref{eq:appV} satisfies $V \in E^{m+2}$, so $\partial_{\nu}\tilde{v}=0$ on $\Sigma$. Thus, in the domain $\left((0,t_1] \cup [t_2,T)\right) \times \Omega$, the solution $\tilde{v} \in H_0$ to \eqref{eq:tilde{v}} satisfies
\begin{equation*}
\left\{
\begin{aligned}
&(\partial_{tt} - \Delta) \tilde{v} + q \tilde{v} = 0 &&\text{in } \left((0,t_1] \cup [t_2,T)\right) \times \Omega, \\
&\tilde{v} = \partial_\nu \tilde{v}=0 && \text{on } \Sigma, \\
&\tilde{v} =\partial_t \tilde{v} = 0 && \text{on } \{t=0\}\times\Omega.
\end{aligned}
\right.
\end{equation*}
By the uniqueness of solutions to wave equations, we have
\begin{equation*}
    \tilde{v} \equiv 0 \quad \textrm{in } (0,t_1)\times \Omega.
\end{equation*}
By the observability result in Lemma \ref{Lemma3.1}, we have
\begin{equation*}
    \tilde{v} \equiv 0 \quad \textrm{in } (t_2,T)\times \Omega.
\end{equation*}
Hence,
\begin{equation*}
\left\{
\begin{aligned}
&\tilde{v}(\cdot,t_1) = \partial_t \tilde{v}(\cdot,t_1)=\tilde{v}(\cdot,t_2) = \partial_t \tilde{v}(\cdot,t_2)  &&\text{in } \Omega, \\
&\tilde{v}=\partial_{\nu}\tilde{v} =0 && \text{on } \Sigma.
\end{aligned}
\right.
\end{equation*}
It follows that
\begin{equation*}
    \int_{t_1}^{t_2}\int_{\Omega} fv dxdt= \int_{t_1}^{t_2}\int_{\Omega} (\partial_{tt} \tilde{v} - \Delta \tilde{v} + q \tilde{v})v dxdt =0
\end{equation*}
for any $v\in Y$ as desired.

\end{proof}

To convey the main idea of the proof, we first consider the quadratic case of Theorem \ref{thm1.1}.

\subsection{Quadratic case}\label{subsec:quad}
In this case,  we consider the inverse problem for 
\[
\partial_{tt} u - \Delta u + q_1 u + q_2 u^2= F \quad \textrm{in } Q.
\]
By assumption, there is an open set $\mathcal N\subset O_{m+1}$ such that 
\begin{align*}
\Lambda_{\vec q_1,F_1,g,h}(f)=\Lambda_{\vec q_2,F_2,g,h}(f) \quad \text{for any} \ f \in \mathcal N.
\end{align*}
Let $f_0 \in \mathcal N$, $f_1 \in O_{m+1}$ and $\epsilon >0$ small enough such that $f:=f_0 + \epsilon f_1 \in \mathcal N$. 

We apply the first-order linearization to the equation 
\begin{equation}\label{eq:qua_uj}
\left\{
\begin{aligned}
&(\partial_{tt}-\Delta)u_j + q_{1,j} u_j + q_{2,j} u_j^2 = F_j && \text{in } Q, \\
&u_j = f && \text{on } \Sigma, \\
&u_j =g,\quad \partial_t u_j = h && \text{on } \left\{t=0\right\}\times\Omega. 
\end{aligned}
\right.
\end{equation}
We denote $u_j^{(0)}$ the solution to 
\begin{equation*}
\left\{
\begin{aligned}
&(\partial_{tt} - \Delta) u_j^{(0)} + q_{1,j} u_j^{(0)} + q_{2,j} \left(u_j^{(0)}\right)^2 = F_j &&\text{in } Q, \\
&u_j^{(0)} = f_0 && \text{on } \Sigma, \\
&u_j^{(0)} =g, \quad \partial_t u_j^{(0)} = h && \text{on } \{t=0\}\times\Omega.
\end{aligned}
\right.
\end{equation*}
With well-posedness holding on a neighborhood of $f_0$, we can differentiate with respect to $\epsilon$ to obtain
\begin{equation}\label{eq:u_j1}
\left\{
\begin{aligned}
&(\partial_{tt}-\Delta)u_j^{(1)}+\left(q_{1,j}+2q_{2,j}u_j^{(0)}\right)u_{j}^{(1)}=0   &&  \text{in } Q,  \\
&u_j^{(1)}=f_1 &&  \text{on } \Sigma,\\
&u_j^{(1)}= \partial_t u_j^{(1)} =0 && \text{on } \left\{  t=0 \right\}\times\Omega.  
\end{aligned}
\right.
\end{equation}
for $u_j^{(1)}:=\frac{\partial u_j}{\partial \epsilon}\big|_{\epsilon=0}$. By the assumption that the DtN maps coincide, we obtain that the linearized DtN maps
\[D\Lambda_{\vec q_j,F_j,g,h}[f_0](f_1):=\frac{\delta}{\delta \epsilon}\Lambda_{\vec q_j,F_j,g,h}(f_0+\epsilon f_1)\big|_{\epsilon=0}=\partial_\nu u^{(1)}_j\big|_{\Sigma}\]
are identical for $j=1,2$. For  convenience, we set 
\begin{align*}
&\mathfrak q_j=q_{1,j}+2q_{2,j} u_{j}^{(0)}, \\
&v=u_1^{(1)}-u_2^{(1)}, \tilde{T}=\frac{T-T^*+t_2}{2},\tilde{Q}=(0,\tilde{T}) \times \Omega.
\end{align*}
Then  $v$ in  $(\tilde{T},T) \times \Omega$ satisfies
\begin{equation}
\left\{
\begin{aligned}
&\partial_{tt}v-\Delta v=0 \qquad &&\textrm{in } (\tilde{T},T)\times\Omega,\\
&v=0 && \textrm{on }(\tilde{T},T)\times\partial\Omega,\\
&v=u_1^{(1)}(T,x)-u_2^{(1)}(T,x), \partial_t v=\partial_t u_1^{(1)}(T,x)-\partial_tu_2^{(1)}(T,x) && \textrm{on } \{t=\tilde{T}\}\times\Omega. 
\end{aligned}
\right.
\end{equation}
By the observability result in Lemma \ref{Lemma3.1}, we can derive that $u_1(\tilde{T},x)=u_2(\tilde{T},x)$ and $\partial_t u_1(\tilde{T},x)=\partial_t u_2(\tilde{T},x)$ in $\Omega$. Combine with the same linearized DtN map, one has 
\begin{align*}
&v(0,x)=0,\ \partial_tv(0,x)=0    \quad &&\text{in} \ \Omega, \\
&v(\tilde{T},x)=0,\ \partial_tv(\tilde{T},x)=0  && \text{in} \  \Omega, \\
& v=0,\quad \partial_{\nu}v=0     && \text{on} \ \Sigma.
\end{align*}
Then  $v$ in  $\tilde{Q}$ satisfies
\begin{equation} \label{eq3.7}
\left\{
\begin{aligned}
&\partial_{tt}v-\Delta v+\mathfrak{q}_1 v=(\mathfrak{q_2-q_1})u_2 \qquad &&\textrm{in }  \tilde{Q},\\
&v=0 && \textrm{on }(0,\tilde{T})\times\partial\Omega,\\
&v= \partial_t v=0 && \textrm{on } \{t=0\}\times\Omega. 
\end{aligned}
\right.
\end{equation}
Let $\tilde{u}_1 \in C^2 \left(\overline{\tilde{Q}} \right)$ be a solution 
\begin{equation}
\partial_{tt}\tilde{u}_1-\Delta \tilde{u}_1+\mathfrak{q_1} \tilde{u}_1=0 \   \text{ in }\tilde{Q}.
\end{equation}
Multiplying both sides of the first equation in \eqref{eq3.7} by $\tilde{u}_1$, integration by parts yields that 
\begin{equation}
\int_{\tilde{Q}}(\mathfrak{q_2-q_1})\tilde{u}_1 u_2 dx dt =0.
\end{equation}
By assumption on the support of $\mathfrak{q_j}$, this implies 
\begin{equation}
\int_{t_1}^{t_2} \int_{\Omega}(\mathfrak{q_2-q_1})\tilde{u}_1 u_2 dx dt =0.
\end{equation}
 Following the argument in \cite{lin_determining_2024}, one can construct the geometrical optics (GO) solutions $v_j$ to be the GO solutions to 
 \begin{equation}
\partial_{tt}v_j -\Delta v_j + \mathfrak{q_j}v_j=0 \ \text{in} \ (t_1,t_2) \times \Omega, 
 \end{equation}
 with the form:
\[v_1(t,x)=e^{-i\tau[\psi(x)+t]}a_1(t,x)+R_1(t,x),\quad v_2(t,x)=e^{i\tau[\psi(x)+t]}a_2(t,x)+R_2(t,x),\]
where $\tau \in \mathbb{R}$ with $|\tau|>1$, $\psi(x)=|x-x_0|$ for some  $x_0 \in \overline{\Omega}$, $a(t,x)$ satisfies 
\[2a_t-2\nabla\psi\cdot\nabla a-\Delta\psi a=0\qquad \textrm{in }(t_1, t_2)\times\Omega,\]
and $R_\tau(t,x)$ satisfies 
\begin{equation*}
\left\{
\begin{aligned}
&(\partial_{tt}-\Delta+\mathfrak q)R_\tau=-e^{i\tau(\psi(x)+t)}\left(\partial_{tt}-\Delta+\mathfrak q\right)a \qquad &&\textrm{in } (t_1,t_2)\times\Omega,\\
&R_\tau=0 && \textrm{on }(t_1,t_2)\times\partial\Omega,\\
&R_\tau=\partial_t R_\tau=0 && \textrm{on } \{t=t_1\}\times\Omega \textrm{ or } \{t=t_2\}\times\Omega,
\end{aligned}
\right.
\end{equation*}
and the asymptotic decaying property
\[\lim_{|\tau|\rightarrow\infty}\|R_\tau\|_{L^2((t_1,t_2)\times\Omega)}=0.\]
By the approximation result in Lemma \eqref{lem:approx}, there are two sequences of complex-valued functions $u_k^1$ and $u_k^2$ such that for $j=1,2$, $u_k^j \in C^2 \left( \overline {\tilde{Q}} \right)$ is the solution to 
\begin{equation*}
\left\{
\begin{aligned}
&(\partial_{tt}-\Delta+\mathfrak q_j)u_k^j=0 \qquad &&\textrm{in } \tilde{Q},\\
&u_k^j=\partial_t u_k^j=0 && \textrm{on } \{t=0 \} \times\Omega,
\end{aligned}
\right.
\end{equation*}
and 
\begin{equation*}
u_k^j \to v_j \quad \text{in}  \ L^2((t_1,t_2) \times \Omega).
\end{equation*}
Then we can choose $\tilde{u}_1=u_k^1$ and $u_2=u_k^2$,  and let $k$ tends to $\infty$, then we obtain
\begin{equation}
\int_{t_1}^{t_2}\int_{\Omega}(\mathfrak{q_2-q_1})v_1 v_2 dx dt =0.
\end{equation}
By applying the similar arguments in \cite{kian_recovery_2019},
we can derive that $\mathfrak{q_1=q_2 } \   \text{in} \  Q$ and we set
\begin{align}\label{eq:Q1}
\mathfrak{q}=q_{1,1}+2q_{2,1} u_1^{(0)}=q_{1,2}+2q_{2,2}u_2^{(0)} \quad \text{in} \  Q. 
\end{align}
Next we apply the second-order linearization. We consider Dirichlet data $f=f_0+\epsilon_1 f_1 +\epsilon_2 f_2$ where $f_0 \in \mathcal N$, $f_1, f_2 \in O^{m+1}$ and $\epsilon_1$ and $\epsilon_2$ are real numbers small enough such that $f=f_0+\epsilon_1 f_1 +\epsilon_2 f_2 \in \mathcal N$. 
Still we still denote by $u_j$ the solution of \eqref{eq:qua_uj} and 
\[
w_{j}=\left.\frac{\partial^2 u_j}{\partial \epsilon_1 \partial \epsilon_2}\right|_{\epsilon_1=\epsilon_2=0} \qquad j=1,2.
\]
Then $w_j$ solves
\begin{equation*}
\left\{
\begin{aligned}
&(\partial_{tt}-\Delta)w_j+\mathfrak q w_j +2q_{2,j}u_{j}^{(1),1} u_{j}^{(1),2}=0 &&  \text{in} \  Q,  \\
&w_j =0 &&  \text{on } \Sigma, \\
&w_j=\partial_t w_j =0  &&  \text{on} \left\{  t=0 \right \}\times\Omega, 
\end{aligned}
\right.
\end{equation*}
where  $u_{j}^{(1),\ell}$, $\ell=1,2$ denotes the solution to the linear equation
\begin{equation*}
\left\{
\begin{aligned}
&(\partial_{tt}-\Delta)u_{j}^{(1),\ell}+\mathfrak q u_{j}^{(1),\ell}=0    &&  \text{in }  Q,   \\
&u_{j}^{(1),\ell}=f_\ell &&  \text{on } \Sigma, \\
&u_{j}^{(1),\ell}= \partial_t u_{j}^{(1),\ell} =0 && \text{on} \left\{  t=0 \right\}\times\Omega. 
\end{aligned}
\right.
\end{equation*}
By the uniqueness of solutions to the Dirichlet problem, we have
\begin{align*}
u^{(1),\ell}:=u_{1}^{(1),\ell}=u_{2}^{(1),\ell} \quad \text{in} \ Q, \quad\text{for} \ \ell=1,2.
\end{align*}
Furthermore, for $f_3\in O_{m+1}$, let $u_3$ solve the backward wave equation
\begin{equation*}
\left\{
\begin{aligned}
&(\partial_{tt}-\Delta)u_3+\mathfrak q u_3=0   \quad  &&\text{in} \ Q,    \\
&u_3=f_3 \quad && \text{on } \Sigma,  \\
&u_3= \partial_t u_3 =0 \quad && \text{on} \left\{  t=T \right\}\times \Omega.  
\end{aligned}
\right.
\end{equation*}
From assumption, we have 
\[\partial_\nu w_1|_\Sigma=\partial_\nu w_2|_\Sigma.\]
Integration by parts yields
\begin{equation}\label{eq:IBP}
\begin{aligned}
&\int_{0}^{T} \int_{\partial \Omega} \left.\frac{\partial^2}{\partial \epsilon_1 \partial \epsilon_2} \Lambda_{\vec q_j,F_j,g_j,h_j}(f)\right|_{\epsilon_1=\epsilon_2=0} \ f_3 \ dS dt \\
= &\int_0^T \int_{\partial \Omega} \partial_{\nu} w_j f_3 \ dS dt \\
=& \int_{0}^{T} \int_{\Omega} \Delta w_j u_3 + \nabla w_j \cdot \nabla u_3 \ dxdt \\
=& \int_0^T \int_{\Omega} \Delta w_j u_3 -  w_j \Delta u_3 \ dx dt \\
=& \int_0^T \int_{\Omega}(\partial_{tt} w_j +\mathfrak q w_j +2 q_{2,j}u^{(1),1} u^{(1),2})u_3 - w_j \Delta u_3 \ dx dt \\
=& \int_0^T \int_{\Omega} w_j(\partial_{tt}-\Delta + \mathfrak q)u_3+2 q_{2,j} u^{(1),1} u^{(1),2} u_3 \ dx dt \\
=& \int_0^T \int_{ \Omega} 2q_{2,j} u^{(1),1} u^{(1),2} u_3 \ dx dt.
\end{aligned}
\end{equation}
Then we have 
\begin{align}\label{eq:2nd_integral}
\int_0^T \int_{\Omega} (q_{2,2}-q_{2,1})u^{(1),1}u^{(1),2} u_3 \ dx dt =0.
\end{align}
By assumption on the support of $q_{i,j}$, this implies 
\[\int_{t_1}^{t_2}\int_\Omega(q_{2,2}-q_{2,1})u^{(1),1}u^{(1),2} u_3 \ dx dt =0.\]
The rest of the proof follows by the same argument as in \cite{lin_determining_2024}, which essentially implements the approximation property and the denseness of the product of the above GO solutions. More specifically, by Lemma \ref{lem:approx}, the integral identity \eqref{eq:2nd_integral} implies
\[\int_{t_1}^{t_2}\int_\Omega (q_{2,2}-q_{2,1})v_1v_2u_3\ dxdt=0\]
for $v_1$ and $v_2$ to be the GO solutions 
\[v_1(t,x)=e^{-i\tau[\psi(x)+t]}a_1(t,x)+R_1(t,x),\quad v_2(t,x)=e^{i\tau[\psi(x)+t]}a_2(t,x)+R_2(t,x),\]
where $\tau \in \mathbb{R}$ with $|\tau|>1$, $\psi(x)=|x-x_0|$ for some  $x_0 \in \overline{\Omega}$, $a(t,x)$ satisfies 
\[2a_t-2\nabla\psi\cdot\nabla a-\Delta\psi a=0\qquad \textrm{in }(t_1, t_2)\times\Omega,\]
and $R_\tau(t,x)$ has the asymptotic decaying property
\[\lim_{|\tau|\rightarrow\infty}\|R_\tau\|_{L^2((t_1,t_2)\times\Omega)}=0.\]
This leads to
\[(q_{2,2}-q_{2,1})u_3=0\qquad \textrm{ in } (t_1,t_2)\times\Omega,\]
which further implies
\[\int_{t_1}^{t_2}\int_\Omega (q_{2,2}-q_{2,1})u_3 v\ dxdt=0\]
for $v$ solving
\begin{equation*}
\left\{
\begin{aligned}
&(\partial_{tt}-\Delta+\mathfrak q) v=0   \quad  &&\text{in} \ Q,    \\
&v= \partial_t v =0 \quad && \text{on} \left\{  t=T \right\}\times \Omega.  
\end{aligned}
\right.
\end{equation*}
Apply above argument again, we eventually obtain 
\[q_{2,2}=q_{2,1}\qquad \textrm{in } (t_1,t_2)\times\Omega.\]
In the following, we denote $q_{2,j}$ by $q_2$.
Set 
\begin{align}\label{eq:psi}
\psi =u_1^{(0)}-u_2^{(0)}\in E^{m+1}.
\end{align}
Plugging into \eqref{eq:Q1}, we get 
\begin{align}
q_{1,1} =q_{1,2}+2q_2 \psi.
\end{align}
Furthermore, calculation leads to 
\begin{align*}
F_2= &(\partial_{tt}-\Delta)u_2^{(0)}+q_{1,2}u_2^{(0)}+q_2 \left(u_2^{(0)}\right)^2 \\
=&(\partial_{tt}-\Delta)\left(u_1^{(0)}+\psi \right)+q_{1,2}\left(u_1^{(0)} +\psi\right)+q_2 \left(u_1^{(0)} +\psi \right)^2 \\
=& F_1+(\partial_{tt}-\Delta) \psi +q_{1,2} \psi +q_2 \psi^2.
\end{align*}
The above formula suggests that $q_{1,2}-q_{1,1}$ and $F_2-F_1$ only depend on terms of $\psi$, with no term that is specifically dependent on $u_j^{(0)}$. In other words, the differences are independent of the choice of $f_0$. Therefore, this proves the conclusion of Theorem \ref{thm1.1} for the case $n=2$. 
\subsection{General case}
 We first prove Theorem \ref{lem_1.3}.

\begin{proof}[Proof of Theorem \ref{lem_1.3}]
The proof is by induction on the order of differentiation $k \in  \mathbb{N}$. Let us consider the Dirichlet data of the form 
\begin{align*}
f:=f_0+ \sum \limits_{l=1}^n \epsilon_l f_l,
\end{align*}
where $f_0 \in \mathcal N$, $f_l \in O^{m+1}$ and $\epsilon_l >0$ are small parameters such that $f_0 +\sum\limits_{l=1}^n \epsilon_lf_l \in \mathcal N$.
We denote $\vec{\epsilon}=(\epsilon_1,\epsilon_2,\dots,\epsilon_n)$, which means that $\vec\epsilon=0$ is equivalent to $\epsilon_1=\epsilon_2=...=\epsilon_n=0$. Let $u_{0,j}$ be the solution to the equation
\begin{equation*}
\left\{
\begin{aligned}
    &(\partial_{tt} - \Delta) u_{0,j} +a_j(t,x,u_{0,j}) = F_j &&\text{in } Q, \\
&u_{0,j} = f_0 && \text{on } \Sigma, \\
&u_{0,j} =g,\quad \partial_t u_{0,j} = h && \text{on } \{t=0\}\times\Omega. 
\end{aligned}
\right.
\end{equation*}
Denote $u_j^{(1),l}=\partial_{\epsilon_l}u_j|_{\vec\epsilon=0}$, then $u_j^{(1),l}$ solves
\begin{equation}\label{eq:u_j^(1)}
\left\{
\begin{aligned}
    &\left(\partial_{tt} - \Delta+\partial_za_j(t,x,u_{0,j})\right) u_j^{(1),l} = 0 &&\text{in } Q, \\
&u_j^{(1),l} = f_l && \text{on } \Sigma, \\
&u_j^{(1),l} = \partial_t u_j^{(1),l} = 0 && \text{on } \{t=0\}\times\Omega. 
\end{aligned}
\right.
\end{equation}
for $j=1,2$, and $l=1,2,\dots,n$. Therefore, by the uniqueness result for the linear inverse boundary value problems, we have
\begin{align*}
    \mathfrak q_1:=\partial_za_1(t,x,u_{0,1})=\partial_za_2(t,x,u_{0,2}) \quad \text{in} \  Q.
\end{align*}
Moreover, by the uniqueness of solutions to the Dirichlet problem \eqref{eq:u_j^(1)}, we have
\begin{align*}
    u^{(1),l}:=u_1^{(1),l}=u_2^{(1),l}
\end{align*}
for $l=1,2,\dots,n$.
Let $\vec\ell=(\ell_1,\ldots,\ell_n)\in\{0,1\}^n$. 
For $k=2,\ldots,n$, denote
\[u_j^{(k),\vec\ell}:=\partial_{\vec\epsilon}^{\vec\ell} u_j\big|_{\vec\epsilon=0}=\partial_{\epsilon_1}^{\ell_1}\partial_{\epsilon_2}^{\ell_2}\cdots \partial_{\epsilon_n}^{\ell_n}u_j\big|_{\vec\epsilon=0},\quad\textrm{ for }\ \vec\ell\in\{0,1\}^n,\ |\vec\ell|=\sum_{m=1}^n \ell_m=k.\]
It is first verified that for $k=2$, we have that $u_j^{(2),\vec\ell}$ with $|\vec\ell|=0$ satisfies
\begin{equation}\label{eq:u_j^k2)}
\left\{
\begin{aligned}
    &(\partial_{tt} - \Delta+ \mathfrak q_1)u_j^{(2),\vec\ell}+\partial_z^2 a_j(t,x,u_{0,j})\prod_{\ell_m\neq0} u^{(1),\ell_m}= 0 &&\text{in } Q, \\
&u_j^{(2),\vec\ell} = 0 && \text{on } \Sigma, \\
&u_j^{(2),\vec\ell} = \partial_t u_j^{(2),\vec\ell} = 0 && \text{on } \{t=0\}\times\Omega.
\end{aligned}
\right.
\end{equation}
(Here note that in order to determine $\partial_z^2 a_j(t,x,u_{0,j})$, it is sufficient to use just $\vec\ell=(1,1,0,\ldots,0)$, i.e., consider only $\partial_{\epsilon_1}\partial_{\epsilon_2}$ term. However, for the induction, we need solutions $u^{(2),\vec\ell}$ for all $|\vec\ell|=2$.) Then based on the previous demonstrating argument for the quadratic nonlinear case in \S\ \ref{subsec:quad}, one immediately obtains 
\[\partial_z^2a_1(t,x,u_{0,1})=\partial_z^2a_2(t,x,u_{0,2})=:\mathfrak q_2 \qquad \textrm{ in }Q.\]
Consequently, we obtain that the solutions satisfy
\[u_1^{(2),\vec\ell}=u_2^{(2),\vec\ell}=:u^{(2),\vec\ell}\qquad\textrm{ for }\vec\ell\in\{0,1\}^n,\ |\vec\ell|=2.\]
For induction, suppose that the first $k\leq n$ partials of $a_j$ and the solutions $u_j$ are identical, that is, for $0<k'\leq k$, $|\vec\ell|=k'$,
\[\partial_z^{k'}a_1(\cdot,\cdot,u_{0,1})=\partial_z^{k'}a_2(\cdot,\cdot,u_{0,2})=:\mathfrak q_{k'},\qquad u^{(k'),\vec\ell}_1=u^{(k'),\vec\ell}_2=:u^{(k'),\vec\ell}.\]
Taking the \( k + 1 \)-th order linearization leads to the equation 
\begin{equation}\label{eq:k+1}
\left\{
\begin{aligned}
&(\partial_{tt}-\Delta+\mathfrak q_1)u_j^{(k+1),\vec\ell}=\mathcal M(t,x)-\partial_z^{k+1}a_j(t,x,u_{0,j})\prod_{\ell_m\neq 0} u^{(1),\ell_m}\qquad &&\textrm{ in }Q,\\
&u_j^{(k+1),\vec\ell}=0 &&\textrm{ on }\Sigma,\\
&u_j^{(k+1),\vec\ell}=\partial_tu_j^{(k+1),\vec\ell}=0 && \textrm{ on }\{t=0\}\times\Omega.
\end{aligned}
\right.
\end{equation}
Here, the product $\prod_{\ell_m\neq 0} u^{(1),\ell_m}$ is over all non-zero components $\ell_m$ of $\vec\ell$ in $u^{(k+1),\vec\ell}$. Also, $\mathcal M(t,x)$ collects all terms containing $\mathfrak q_1, \ldots, \mathfrak q_k$ and $u^{(1),\ell}$, $u^{(2),\vec\ell}, \ldots, u^{(k),\vec\ell}$, and therefore is a known term in the equation.
Similarly, multiplying the equation in \eqref{eq:k+1} by a backward linear solution $v$ of
\[
\left\{
\begin{aligned}
&(\partial_{tt}-\Delta+\mathfrak q_1)v=0\qquad&&\textrm{ in }Q,\\
&v=\partial_tv=0 && \textrm{ on }\{t=T\}\times\Omega,
\end{aligned}
\right.
\]
and integration-by-parts as in \eqref{eq:IBP}, we obtain
\begin{equation}\label{eq:int_k+1}
\int_0^T\int_\Omega \left(\partial_z^{k+1}a_2(t,x,u_{0,2})-\partial_z^{k+1}a_1(t,x,u_{0,1})\right)u^{(1),1}u^{(1),2}\cdots u^{(1),k+1}v\ dx dt=0,
\end{equation}
where we only considered the term whose $\vec\ell=(1,\ldots,1,0,\ldots,0)$ with the first $k+1$ components being 1. Note that the known term $\mathcal M(t,x)$ will cancel when taking the difference after the integration-by-parts step. Following the approximation to the GO solutions argument as in \S \ref{subsec:quad}, we eventually prove
\begin{equation} \label{partial}
\partial_z^{k+1}a_1(t,x,u_{0,1})=\partial_z^{k+1}a_2(t,x,u_{0,2}).
\end{equation}
This completes the proof by induction of the theorem.
\end{proof}

\begin{proof}[Proof of Theorem \ref{thm1.1}]
By the observability inequality in Lemma \ref{Lemma3.1}, there exists a constant \(C > 0\) such that
\begin{equation*}
\|g_1 - g_2\|_{H_0^1(\Omega)} + \|h_1 - h_2\|_{L^2(\Omega)} \leq C \|\partial_{\nu}(u_1 - u_2)\|_{L^2((0,t_1) \times \partial\Omega)} = 0.
\end{equation*}
This implies \(g_1 = g_2\) and \(h_1 = h_2\) in \(\Omega\).\\
For any $q_{i,j}\in E^{m+1}$ and $F_j\in E^{m-1}$ for $i=1,2,\dots,n$, $j=1,2$. Set $\psi=u_{0,2}-u_{0,1}$, then $\psi \in E^{m+1}$ with $\psi|_{\Sigma}= \partial_{\nu} \psi|_{\Sigma}=0$. We need to show that
\begin{align}\label{eq:q_n-i}
    q_{n-i,1}=\sum\limits_{l=0}^i \binom{n-i+l}{l}q_{n-i+l,2}\psi^l
\end{align}
for $i=0,1,\dots,n-1$. By Theorem \ref{lem_1.3}, we have
\begin{align*}
    q_{n,1}=\frac{1}{n!} \partial_z^n a_1(t,x,u_{0,1})=\frac{1}{n!} \partial_z^n a_2(t,x,u_{0,2})=q_{n,2}\quad \text{in} \  (t_1,t_2)\times \Omega.
\end{align*}
Thus, the claim holds for $i=0$. We prove the claim by induction. Let us assume that \eqref{eq:q_n-i} holds for $i=0,1,\dots,k$. It suffices to show that \eqref{eq:q_n-i} holds for $i=k+1$.\\
Using Theorem \ref{lem_1.3} again, we have 
\begin{align*}
    \partial_z^{n-(k+1)} a_1(t,x,u_{0,1})&= \sum\limits_{i=0}^{k+1} \frac{(n-i)!}{(k+1-i)!}q_{n-i,1}u_{0,1}^{k+1-i}\\
    &=\sum\limits_{i=0}^{k+1} \frac{(n-i)!}{(k+1-i)!}q_{n-i,2}u_{0,2}^{k+1-i}= \partial_z^{n-(k+1)} a_2(t,x,u_{0,2}) \quad \text{in} \  (t_1,t_2)\times \Omega.
\end{align*}
After dividing by $(n-k-1)!$, the above reads
\begin{align}\label{eq:q_n-(k+1)}
    q_{n-(k+1),1}=q_{n-(k+1),2}+ \sum\limits_{i=0}^k \binom{n-i}{n-k-1}q_{n-i,2} u_{0,2}^{k+1-i}- \sum\limits_{i=0}^k \binom{n-i}{n-k-1}q_{n-i,1}u_{0,1}^{k+1-i}.
\end{align}
Using the induction assumption and the binomial expansion, we consider the coefficients of the term $(u_{0,1})^J$. We observe that the coefficient of $(u_{0,1})^J$ is
\begin{align}\label{eq:S_J}
S_J 
&= \sum_{i=0}^k \binom{n-i}{n-k-1} q_{n-i,2} \binom{k+1-i}{k+1-i-J} \psi^{k+1-i-J} \notag\\
&\quad 
- \binom{n-k-1+J}{n-k-1} \sum_{l=0}^{k+1-J} \binom{n-k-1+J+l}{l} q_{n-k-1+J+l,1} \psi^l \notag\\
&= \sum_{l=0}^{k+1-J} \binom{n-k-1+J+l}{n-k-1} q_{n-k-1+J+l,2} \binom{J+l}{l} \psi^l \notag\\
&\quad
- \binom{n-k-1+J}{n-k-1} \sum_{l=0}^{k+1-J} \binom{n-k-1+J+l}{l} q_{n-k-1+J+l,2} \psi^l
\end{align}
for $J=0,1,...,k+1$. A direct computation shows that
\begin{align*}
    \binom{n-k-1+J+l}{n-k-1}\binom{J+l}{l}=\binom{n-k-1+J}{n-k-1}\binom{n-k-1+J+l}{l}
\end{align*}
for $l=0,1,...,k+1-J$. Thus the term of $(u_{0,1})^J$ in \eqref{eq:q_n-(k+1)} is zero for $J=1,2,...,k+1$.\\
As for the zeroth power of $u_{0,1}$ in \eqref{eq:q_n-(k+1)}, we have
\begin{align}\label{eq:S0}
    S_0:=\sum\limits_{i=0}^k \binom{n-i}{n-k-1} q_{n-i,2} \psi^{k+1-i} =\sum\limits_{l=1}^{k+1} \binom{n-(k+1)+l}{l} q_{n-(k+1)+l,2} \psi^l.
\end{align}
Therefore, by plugging \eqref{eq:S_J} and \eqref{eq:S0} into \eqref{eq:q_n-(k+1)}, we have
\begin{align*}
    q_{n-(k+1),1}= q_{n-(k+1),2}+ \sum\limits_{l=1}^{k+1} \binom{n-(k+1)+l}{l} q_{n-(k+1)+l,2} \psi^l =\sum\limits_{l=0}^{k+1} \binom{n-(k+1)+l}{l} q_{n-(k+1)+l,2} \psi^l.
\end{align*}
This proves the induction step. It remains to prove the last equality of \eqref{eq:thm1_result}.\\
Based on the conclusion of \eqref{eq:q_n-i}, we can write
\begin{align*}
    a_1(t,x,u_{0,1}) = \sum\limits_{i=1}^n q_{i,1}u_{0,1}^i = \sum\limits_{i=0}^{n-1}q_{n-i,1} u_{0,1}^{n-i} = \sum\limits_{i=0}^{n-1} \sum\limits_{l=0}^i \binom{n-i+l}{l}q_{n-i+l,2}\psi^lu_{0,1}^{n-i}.
\end{align*}\\
On the other hand, we can express
\begin{align*}
    a_2(t,x,u_{0,2})  = \sum\limits_{i=0}^{n-1}q_{n-i,2} (u_{0,1}+\psi)^{n-i} = \sum\limits_{i=0}^{n-1} q_{n-i,2} \sum\limits_{m=0}^{n-i} \binom{n-i}{m} \psi^m u_{0,1}^{n-i-m}.
\end{align*}
Similar to the computation of (3.15), the coefficient of $(u_{0,1})^{n-J}$ in $a_2(t,x,u_{0,2})$ is 
\begin{align*}
    \sum\limits_{i=0}^{n-1} q_{n-i,2} \binom{n-i}{J-i} \psi^{J-i}= \sum\limits_{l=0}^Jq_{n-J+l,2} \binom{n-J+l}{l} \psi^l
\end{align*}
for ${J=0,1,...,n}$, which is equal to the coefficient of $(u_{0,1})^{n-J}$ in $a_1(t,x,u_{0,1})$ for ${J=0,1,\dots,n-1}$.\\
Thus, we have
\begin{align*}
    a_1(t,x,u_{0,1})-a_2(t,x,u_2)=-\sum \limits_{l=1}^n q_{l,2} \psi^l.
\end{align*}
Therefore,
\begin{align*}
    F_1 = F_2+(\partial_{tt} - \Delta)(u_{0,1} - u_{0,2}) + a_1(t,x,u_{0,1}) - a_2(t,x,u_{0,2})= F_2- (\partial_{tt} - \Delta) \psi-\sum\limits_{l=1}^nq_{l,2}\psi^l.
\end{align*}
Conversely, if there exist $\psi \in E^{m+1}$ with $\psi|_{\Sigma}=\partial_{\nu} \psi|_{\Sigma}=0$ makes (1.4) valid.\\
Let $u_1$ solve (1.2) for $j=1$, define
\begin{equation*}
    w:=u_1+\psi.
\end{equation*}
Then we have $(w|_{\Sigma},\partial_{\nu}w|_{\Sigma})=(u_1|_{\Sigma},\partial_{\nu} u_1|_{\Sigma})$, and
\begin{align*}
    (\partial_{tt} - \Delta) w + \sum_{i=1}^{n} q_{i,2} w^i &=(\partial_{tt}- \Delta) (u_1+\psi) + \sum_{i=1}^{n} q_{i,2} \left(\sum\limits_{m=0}^i \binom{i}{m} \psi^{i-m} u_1^{m}\right)\\
    &= (\partial_{tt}- \Delta) u_1 + \sum\limits_{m=1}^n \left(\sum\limits_{i=m}^n q_{i,2} \binom{i}{m} \psi^{i-m}\right) u_1^m+ (\partial_{tt}- \Delta)\psi +\sum\limits_{i=1}^nq_{i,2} \psi^i \\
    &=F_2.
\end{align*}
Hence $w$ solves $\partial_{tt} w - \Delta w + \sum_{i=1}^{n} q_{i,2} w^i = F_2$. Since $u_1$ and $w$ also have the same Cauchy data on $\partial \Omega$, it follows that the corresponding DtN maps are the same: $\Lambda_{\vec q_1,F_1,g,h}(f)=\Lambda_{\vec q_2,F_2,g,h}(f) \ \  \text{for any} \ f \in \mathcal{N}$. \\ 
\end{proof}

\begin{proof}[Proof of Corollary \ref{corollary1}]
By the observability inequality in Lemma \ref{Lemma3.1}, there exists a constant \(C > 0\) such that
\begin{equation*}
\|g_1 - g_2\|_{H_0^1(\Omega)} + \|h_1 - h_2\|_{L^2(\Omega)} \leq C \|\partial_{\nu}(u_1 - u_2)\|_{L^2((0,t_1) \times \partial\Omega)} = 0.
\end{equation*}
This implies \(g_1 = g_2\) and \(h_1 = h_2\) in \(\Omega\). Denote \(g = g_1 = g_2\) and \(h = h_1 = h_2\).
Let \(\tilde{u}_j\) solve the homogeneous boundary value problem:
\begin{equation}
\left\{
\begin{aligned}
&(\partial_{tt} - \Delta + q_j)\tilde{u}_j = F_j 
&& \text{in } Q, \\
&\tilde{u}_j = 0 
&& \text{on } \Sigma, \\
&\tilde{u}_j = g,\ \partial_t \tilde{u}_j = h
&& \text{on } \{t=0\} \times \Omega. 
\end{aligned}
\right.
\end{equation}
Define \(v_j = u_j - \tilde{u}_j\) and let \(\tilde{\Lambda}_{q_j}\) be the Dirichlet-to-Neumann map for
\begin{equation}
\left\{
\begin{aligned}
&(\partial_{tt} - \Delta + q_j)v_j = 0
&& \text{in } Q, \\
&v_j = f
&& \text{on } \Sigma, \\
&v_j = \partial_t v_j = 0
&& \text{on } \{t=0\} \times \Omega. 
\end{aligned}
\right.
\end{equation}
The difference of the DtN maps satisfies
\begin{equation*}
\tilde{\Lambda}_{q_1}(f) - \tilde{\Lambda}_{q_2}(f) = \partial_\nu (u_1 - \tilde{u}_1) - \partial_\nu (u_2 - \tilde{u}_2) = \partial_\nu (\tilde{u}_2 - \tilde{u}_1) \ \text{for} \ f \in \mathcal{N} .
\end{equation*}
Choosing \(f_0 \in \mathcal{N}\), $f_1 \in O^{m+1} $ and $\epsilon_0$ small enough such that \(f = f_0+\epsilon_0 f_1\). Thus, one has
\begin{align*}
\partial_\nu (\tilde{u}_2 - \tilde{u}_1)&=\tilde{\Lambda}_{q_1}(f_0+\epsilon f_1) - \tilde{\Lambda}_{q_2}(f_0+\epsilon f_1) \\ &=  \tilde{\Lambda}_{q_1}(f_0)-\tilde\Lambda_{q_2}(f_0)+\epsilon\left(\tilde{\Lambda}_{q_1}(f_1)-\tilde{\Lambda}_{q_2}(f_1)\right)\\
&=\partial_\nu (\tilde{u}_2 - \tilde{u}_1) +\epsilon \left(\tilde{\Lambda}_{q_1}(f_1)-\tilde{\Lambda}_{q_2}(f_1)\right)  \ \text{for all} \   f_1 \in O^{m+1} \ \text{and}  \ \epsilon<\epsilon_0.
\end{align*}
It yields that $\tilde{\Lambda}_{q_1}(f) = \tilde{\Lambda}_{q_2}(f)$ for all $f \in O^{m+1}$.
Following the linear case analysis in Section \ref{subsec:quad}, we conclude \(q_1 = q_2 =: q\) in \((t_1, t_2) \times \Omega\).
Define \(\psi = \tilde{u}_1 - \tilde{u}_2\), which satisfies
\begin{equation}
\left\{
\begin{aligned}
&(\partial_{tt} - \Delta + q)\psi = F_1 - F_2 
&& \text{in } Q, \\
&\psi = \partial_\nu \psi = 0 
&& \text{on } \Sigma, \\
&\psi = \partial_t \psi = 0 
&& \text{on } \{t=0\}\times\Omega.
\end{aligned}
\right.
\end{equation}
Given \(supp(q) \subset (t_1, t_2) \times \Omega\) and \(supp(F_j) \subset (t_1, T) \times \overline{\Omega}\), the solution vanishes in \((0, t_1) \times \Omega\), which completes the proof.
\end{proof}

\begin{proof}[Proof of Theorem \ref{thm_2}]
By the observability inequality in Lemma \ref{Lemma3.1}, there exists a constant \(C > 0\) such that
\begin{equation*}
\|g_1 - g_2\|_{H_0^1(\Omega)} + \|h_1 - h_2\|_{L^2(\Omega)} \leq C \|\partial_{\nu}(u_1 - u_2)\|_{L^2((0,t_1) \times \partial\Omega)} = 0.
\end{equation*}
This implies \(g_1 = g_2\) and \(h_1 = h_2\) in \(\Omega\).\\
Under  assumption(1), using \eqref{eq:thm1_result}, one has 
\begin{equation}
q_{n-1,1}=q_{n-1,2}+\binom{n}{1}q_{n,2} \psi  \quad \text{in} \  (t_1,t_2) \times \Omega. 
\end{equation}
Then we can derive that $\psi=0 $ in $(t_1,t_2) \times \Omega$. Then all the coefficients are determined and $F_1=F_2+\partial_{tt} \psi - \Delta \psi$ in $(t_2,T) \times \Omega$.\\
Under  assumption (2), $\psi \in E^{m+1}$ satisfies the equation
\begin{equation}
\left\{
\begin{aligned}
&\partial_{tt} \psi - \Delta \psi + \sum_{i=1}^n q_{i,2} \psi^i = 0 
&& \text{in } Q, \\
&\psi  =\partial_{\nu} \psi= 0
&& \text{on } \Sigma, \\
&\psi = \partial_t \psi = 0 
&& \text{on } \{t=0\}\times\Omega.
\end{aligned}
\right.
\end{equation}
By Sobolev embedding theorem, one has $\psi \in C(Q)$. Then we have $|\partial_{tt} \psi - \Delta \psi| \leq C |\psi| $ in  $Q$. 
Define the energy functional
\begin{equation}
E(t) = \frac{1}{2} \int_{\Omega} \left( (\partial_t \psi(x,t))^2 + |\nabla \psi(x,t)|^2 \right) dx.
\end{equation}
Due to the boundary condition $\psi = 0$ on $\partial \Omega \times [0,T]$, Poincaré's inequality holds: there exists $C_P > 0$ such that
\begin{equation}
\int_{\Omega} \psi(x,t)^2 \, dx \leq C_P \int_{\Omega} |\nabla \psi(x,t)|^2 \, dx.
\end{equation}
Differentiating $E(t)$ with respect to time  yields
\begin{align*}
\frac{dE}{dt} =& \int_{\Omega} \left( \partial_t \psi  \partial_{tt} \psi + \nabla \psi \cdot \nabla \partial_t \psi \right) dx \\
 =&  \int_{\Omega} (\partial_{tt} \psi -\Delta \psi)  \partial_t \psi \, dx + \int_{\partial \Omega} \partial_{\nu} \psi \partial_t \psi \, dS  \\
 =& \int_{\Omega} \partial_t \psi  (\partial_{tt} \psi - \Delta \psi) \, dx.
 \end{align*}
 \begin{align*}
\left| \frac{dE}{dt} \right| \leq \int_{\Omega} |\partial_t \psi| |\partial_{tt} \psi - \Delta \psi| \, dx \leq  & C \int_{\Omega} |\partial_t \psi|  |\psi| \, dx \\
 \leq& \frac{C}{2} \int_{\Omega} \left( (\partial_t \psi)^2 + \psi^2 \right) dx \\
 \leq& \frac{C}{2} \cdot 2(1 + C_P) E(t) = C(1 + C_P) E(t).
\end{align*}
Let $K = C(1 + C_P)$, then
\[
\left| \frac{dE}{dt} \right| \leq K E(t),
\]
At $t = 0$, we have $\psi = 0$ and $\partial_t \psi = 0$ in $\Omega$. Thus we have $E(0) = 0$. From $\frac{dE}{dt} \leq K E(t)$ and $E(0) = 0$, Gronwall's inequality gives
\[
E(t) \leq E(0) e^{K t} = 0 \quad \text{for all} \quad t \in [0,T].
\]
Since $E(t) \geq 0$, it follows that $E(t) = 0$ for all $t \in [0,T]$.
Finally, $E(t) = 0$ implies
 $\partial_t \psi = 0$ and $\nabla \psi = 0$ almost everywhere in $Q$. By the continuity of $\psi$, we conclude that $\psi = 0$  in $Q$.
Then all the coefficients are determined.
\end{proof}
\begin{proof}[Proof of Corollary \ref{corollary apply}]
Let $u_{0,j}$ be the solution to 
\begin{equation} \label{eqaj}
\left\{
\begin{aligned}
&\partial_{tt} u_{0,j} - \Delta u_{0,j} + a_j(t,x,u_{0,j}) = F_j
&& \text{in } Q, \\
&u_{0,j}  = f_0 
&& \text{on } \Sigma, \\
&u_{0,j} =g_0, \partial_t u_{0,j}  = h_0
&&\text{on } \{t=0\}\times\Omega.
\end{aligned}
\right.
\end{equation}
Case 1: The nonlinearity in Case 1 is $a_j(t,x,z)=q_j \mathrm{e}^{z}$.  Using \eqref{partial} with $k=1$, we have
\begin{align} \label{case1}
q_1 \mathrm{e}^{u_{0,1}}=\partial_z a_1(t,x,u_{0,1})=\partial_z a_2(t,x,u_{0,2})=q_2 \mathrm{e}^{u_{0,2}}.
\end{align}
 Then we set $\psi=u_{0,2}-u_{0,1}$. From \eqref{case1}, we have $q_1=q_2 \mathrm{e}^{\psi}$ in $[t_1,t_2]\times \Omega$. Then by using \eqref{eqaj}, we have                 
 \begin{equation*}
F_2-F_1=(\partial_{tt}-\Delta)(u_{0,2}-u_{0,1})+q_2 \mathrm{e}^{u_{0,2}}-q_1 \mathrm{e}^{u_{0,1}}=\partial_{tt} \psi-\Delta \psi \ \ \text{in} \ [t_1,t_2] \times \Omega.
 \end{equation*}
 Case 2 :  The nonlinearity in Case 2 is $ a_j(t,x,z)=q_j \sin(z)$. Using \eqref{partial} with $k=1$, we have
 \begin{equation} \label{eqcos}
q_1 \cos(u_{0,1})=q_2 \cos (u_{0,2}) \ \ \text{in} \ [t_1,t_2] \times \Omega.
 \end{equation} \label{eqsin}
 and with $k=2$,  we have 
  \begin{equation} 
q_1 \sin(u_{0,1})=q_2 \sin (u_{0,2}) \ \  \text{in} \ [t_1,t_2] \times \Omega.
 \end{equation}
 By the Euler identity, \eqref{eqcos} and \eqref{eqsin} is equivalent to 
 \begin{equation}\label{i}
q_1 \mathrm{e}^{i u_{0,1}}=q_2 \mathrm{e} ^{i u_{0,2}}  \ \ \text{in} \ [t_1,t_2] \times \Omega.
 \end{equation}
 We define $\psi = u_{0,2}-u_{0,1}$. Then by using \eqref{eqaj} and \eqref{eqsin}, one has
 \begin{equation*}
F_2-F_1=(\partial_{tt}-\Delta)(u_{0,2}-u_{0,1})+q_2 \sin(u_{0,2})-q_1 \sin(u_{0,1})=\partial_{tt} \psi -\Delta \psi \ \ \text{in} \ [t_1,t_2] \times \Omega,
 \end{equation*}
 and by \eqref{i},
 \begin{equation*}
q_1=q_2 \mathrm{e}^{i \psi}  \ \ \text{in} \ [t_1,t_2] \times \Omega.
 \end{equation*}
 Since $q_1$ and $q_2$ are real-valued functions and $\psi$ is continuous, we have $\mathrm{e}^{i \psi} \equiv 1$ or $\mathrm{e}^{i \psi} \equiv -1$ in $[t_1,t_2]\times \Omega$. \\
 We show by using boundary determination that  $\psi \equiv 0$ in $[t_1,t_2] \times \Omega$. Let $\epsilon$ be a small real number, $f_0\in \mathcal{N}$, $f_1 \in O^{m+1}$ and $f=f_0+\epsilon f_1$. Using the first linearization method around the solution $u_{0,j}$ of \eqref{eqaj}, one  has 
\begin{equation}  \label{eqfirst}
\left\{
\begin{aligned}
&\partial_{tt} u_{j}^{(1)} - \Delta u_{j}^{(1)} +q_j \cos(u_{0,j}) u_j^{(1)} = 0 
&& \text{in } Q, \\
&u_{j}^{(1)}  = f_1 
&& \text{on } \Sigma, \\
&u_{j}^{(1)} = \partial_t u_{j}^{(1)} = 0 
&& \text{on } \{t=0\}\times\Omega.
\end{aligned}
\right.
\end{equation}
 We can apply the same method of coefficient determination for the linear wave equation in Section \ref{subsec:quad}, then we have 
 \begin{equation*}
q_1 \cos(f_0)=q_2 \cos(f_0)   \ \ \text{in} \ [t_1,t_2] \times  \Omega.
 \end{equation*}
  Using $q_j \cos(u_{0,j}) \in C(Q)$, we conclude that $q_1 \cos(f_0)=q_2 \cos(f_0)  \ \text{on} \ [t_1,t_2] \times   \partial \Omega$. We can choose $f_0 \in \mathcal{N}$  properly such that $\cos(f_0(x_0,t_0)) \neq 0$. Thus, we derive that  $q_1=q_2$ and $\mathrm{e}^{i \psi}\equiv 1    \ \text{in} \ [t_1,t_2] \times \Omega$. Since $\psi$ is a constant, we have $\psi =0 \  \text{in} \ [t_1,t_2] \times \Omega$. Thus, we conclude that 
 \begin{equation*}
  q_1=q_2, F_1=F_2 \ \ \text{in} \ [t_1,t_2] \times \Omega.
 \end{equation*}
 Case 3: The nonlinearity in Case 3 is $ a_j(t,x,z)=q_j z \sin z $. Using \eqref{partial} with $k=1,2,3,4$, we have
 \begin{align}
&q_1 \sin(u_{0,1})+q_1 u_1 \cos(u_{0,1})=q_2 \sin(u_{0,2})+q_2 u_2 \cos(u_{0,2}) \ \ \text{in} \ [t_1,t_2] \times \Omega, \label{eq1} \\
&2q_1 \cos(u_{0,1})-q_1u_1 \sin(u_{0,1})= 2q_2 \cos(u_{0,2})-q_2u_2 \sin(u_{0,2}) \ \ \text{in} \ [t_1,t_2] \times \Omega,   \label{eq2} \\
&-3q_1 \sin(u_{0,1})-q_1u_1 \cos(u_{0,1})=  -3q_2 \sin(u_{0,2})-q_2u_2 \cos(u_{0,2})\ \ \text{in} \ [t_1,t_2] \times \Omega, \label{eq3}   \\
&-4q_1 \cos(u_{0,1})+q_1u_1 \sin(u_{0,1})=-4q_2 \cos(u_{0,2})+q_2u_2 \sin(u_{0,2}) \ \ \text{in} \ [t_1,t_2] \times \Omega. \label{eq4}
 \end{align}
 Add equation \eqref{eq1} to equation \eqref{eq3} and add equation \eqref{eq2} to equation \eqref{eq4},  we have
 \begin{equation}
  q_1 \sin(u_{0,1})=q_2 \sin(u_{0,2})  , \ q_1 \cos(u_{0,1})=q_2 \cos (u_{0,2}) \ \ \text{in} \ [t_1,t_2] \times \Omega.
 \end{equation}
 By using the same analysis in the proof of Case 2, we can derive that 
 \begin{equation*}
q_1=q_2, \ F_1=F_2  \ \ \text{in} \ [t_1,t_2] \times \Omega.
 \end{equation*}
 Case  4: The nonlinearity in Case 4 is $ a_j(t,x,z)=p_j \sin z+q_j \mathrm{e}^{z} $. Using \eqref{partial} with $k=1,2,3,4$, we have
 \begin{align}
&p_1 \cos(u_{0,1})+q_1 \mathrm{e}^{u_{0,1}}= p_2 \cos(u_{0,2})+q_2 \mathrm{e}^{u_{0,2}}   \ \ \text{in} \ [t_1,t_2] \times \Omega, \label{e1} \\
&-p_1 \sin(u_{0,1})+q_1 \mathrm{e}^{u_{0,1}}=-p_2 \sin(u_{0,2})+q_2 \mathrm{e}^{u_{0,2}}  \ \ \text{in} \ [t_1,t_2] \times \Omega, \label{e2} \\
& -p_1 \cos(u_{0,1})+q_1 \mathrm{e}^{u_{0,1}}= -p_2 \cos(u_{0,2})+q_2 \mathrm{e}^{u_{0,2}}\ \ \text{in} \ [t_1,t_2] \times \Omega, \label{e3} \\
 &p_1 \sin(u_{0,1})+q_1 \mathrm{e}^{u_{0,1}}=p_2 \sin(u_{0,2})+q_2 \mathrm{e}^{u_{0,2}} \ \ \text{in} \ [t_1,t_2] \times \Omega. \label{e4}
 \end{align} 
 Now considering equation \eqref{e1} minus equation \eqref{e3} and equation \eqref{e4} minus equation \eqref{e2},  it yields that 
 \begin{equation*}
 p_1 \cos(u_{0,1})=p_2 \cos (u_{0,2}), \ p_1 \sin(u_{0,1})=p_2 \sin(u_{0,2})  , \ \ \text{in} \ [t_1,t_2] \times \Omega.
 \end{equation*}
 By using the same analysis in the proof of Case 2, we can derive that 
 \begin{equation*}
p_1=p_2, \ F_1=F_2  \ \ \text{in} \ [t_1,t_2] \times \Omega.
 \end{equation*}
 Case 5 : The nonlinearity in Case 5 is $a_j(t,x,z)=p_j \sin z+q_j \cos z $. Using \eqref{partial} with $k=1,2$, we have
 \begin{align}  
&p_1 \cos(u_{0,1})-q_1 \sin(u_{0,1})=p_2 \cos(u_{0,2})-q_2 \sin(u_{0,2})  \ \ \text{in} \ [t_1,t_2] \times \Omega, \label{t1}\\
&-p_1 \sin(u_{0,1})-q_1 \cos(u_{0,1})= -p_2 \sin(u_{0,2})-q_2 \cos(u_{0,2})   \ \ \text{in} \ [t_1,t_2] \times \Omega.  \label{t2}
 \end{align} 
 We define $\psi =u_{0,2}-u_{0,1}$. From equation \eqref{t1} and \eqref{t2}, we have 
 \begin{equation}
p_2=p_1 \cos\psi+q_1 \sin\psi, \ q_2=-p_1 \sin\psi + q_1 \cos \psi  \ \ \text{in} \ [t_1,t_2] \times \Omega. 
 \end{equation}
 Then by using \eqref{eqaj} and \eqref{t2}, we have
\begin{align*}
&F_2-F_1=\partial_{tt} \psi -\Delta \psi+p_2 \sin(u_{0,2})+q_2 \cos(u_{0,2})-p_1 \sin(u_{0,1})-q_1 \cos(u_{0,1}) \\&=\partial_{tt}\psi -\Delta \psi \ \ \ \text{in} \ [t_1,t_2] \times \Omega. 
\end{align*}
 Case 6 : The nonlinearity in Case 6 is $ a_j(t,x,z)=p_j z \mathrm{e}^{z} + \sum_{k=1}^n q_{k,j}z^k $.Using \eqref{partial} with $k=n+1,n+2$, we have
 \begin{align}  
p_1(u_{0,1}+n+1) \mathrm{e}^{u_{0,1}}=p_2(u_{0,2}+n+1) \mathrm{e}^{u_{0,2}}  \ \ \text{in} \ [t_1,t_2] \times \Omega, \label{s1}  \\
p_1(u_{0,1}+n+2) \mathrm{e}^{u_{0,1}}=p_2(u_{0,2}+n+2)  \mathrm{e}^{u_{0,2}}  \ \ \text{in} \ [t_1,t_2] \times \Omega.  \label{s2}
\end{align}
Combining \eqref{s1} and \eqref{s2}, we obtain
\begin{equation}
p_1 \mathrm{e}^{u_{0,1}}=p_2 \mathrm{e}^{u_{0,2}} \quad \text{and} \quad p_1 u_{0,1} \mathrm{e}^{u_{0,1}}=p_2 u_{0,2} \mathrm{e}^{u_{0,2}}\ \ \text{in} \ [t_1,t_2] \times \Omega.  
\end{equation}
Since $p_j \neq 0 \ \ \text{in} \ [t_1,t_2] \times \Omega $, we have $u_{0,1}=u_{0,2} \ \ \text{in} \ [t_1,t_2] \times \Omega$, which implies $p_1=p_2 \ \ \text{in} \ [t_1,t_2] \times \Omega$. We process by considering  \eqref{partial} with $k=n,n-1,\cdots,1$ in order, then we can derive that $q_{k,1}=q_{k,2}$ for $k=1,2,\cdots,n \ \ \text{in} \ [t_1,t_2] \times \Omega $. Then we complete the proof.
\end{proof}

\bibliographystyle{amsplain}
\bibliography{ref}
\end{document}